\documentclass[11pt]{amsart}
\usepackage{graphicx}
\usepackage{euscript}
\usepackage{amsmath}
\usepackage{amsfonts}
\usepackage{amssymb}
\usepackage{amsthm}
\usepackage{epsfig}
\usepackage{epstopdf}
\usepackage{url}
\usepackage{array}
\usepackage{amssymb}\usepackage{amscd}
\usepackage{epic}
\usepackage{eufrak}
\usepackage{tikz,underscore}
\usepackage{tikz-cd}
\usepackage{bm}
\usepackage{float}

\numberwithin{equation}{section}

\theoremstyle{plain}
\newtheorem{theorem}{Theorem}[section]
\newtheorem{lemma}[theorem]{Lemma}
\newtheorem{proposition}[theorem]{Proposition}

\newtheorem{corollary}[theorem]{Corollary}

\theoremstyle{definition}
\newtheorem{example}[theorem]{Example}
\newtheorem{remark}[theorem]{Remark}
\newtheorem{definition}[theorem]{Definition}

\def\Sp{S^{3}_{-p_{1}, \cdots, -p_{n}}(\mathcal{L})}

\def\Sph{S^{3}_{-p_{1}, \cdots, -p_{n}}(O)}
\def\deg{\textup{deg}}
\def\Z{\mathbb{Z}}
\def\G{\mathfrak{G}}
\def\g{\bm{g}}
\def\L{\mathcal{L}}
\def\H{\mathbb{H}}
\def\F{\mathbb{F}}

\title{Four genera of  links and Heegaard Floer homology }

\author{Beibei Liu}
\address{B.L.: Department of Mathematics, UC Davis, One Shields Avenue, Davis CA 95616, USA}
\email{bxliu@math.ucdavis.edu}

\begin{document}

\begin{abstract}
For  links with vanishing pairwise linking numbers, the link components bound pairwise disjoint surfaces in $B^{4}$. In this paper, we describe the set of genera of such surfaces  in terms of the $h$-function, which is a link invariant from  Heegaard Floer homology. In particular, we  use the $h$-function to give lower bounds for the 4-genus of the link. For $L$-space links, the $h$-function is explicitly determined by Alexander polynomials of the link and sublinks. We show some  $L$-space links where the lower bounds are sharp, and also describe all possible  genera of disjoint surfaces bounded by such links.
\end{abstract}

\maketitle
\section{Introduction}

Let $\L=L_{1}\cup L_{2}\cdots \cup L_{n}$ be an oriented $n$-component link in $S^{3}$ with all  linking numbers $0$. Recall that a link bounds disjointly embedded surfaces in $B^{4}$ if and only if it has vanishing pairwise linking numbers. The \emph{4-genus} of $\L$ is defined as: 
$$g_{4}(\L)=\min\lbrace \sum\limits_{i=1}^{n} g_{i} \mid g_{i}=g(\Sigma_{i}), \Sigma_{1}\sqcup \cdots \sqcup \Sigma_{n}\hookrightarrow B^{4}, \partial \Sigma_{i}=L_{i} \rbrace .$$
If $\L$ is a knot, the 4-genus is also known as the slice genus. Powell, Murasugi and Livingston showed lower bounds for the $4$-genera of links in terms of the Levine-Tristram signatures, see \cite{Liv, Mur, Po}.  Rasmussen defined the $h$-function (as an analogue of the Fr$\o$yshov invariant in Seiberg-Witten theory) for knots, and used it to obtain nontrivial lower bounds for the slice genus of a knot \cite{Ras1, Ras}. We generalize Rasmussen's result and obtain lower bounds for the 4-genera of  links with vanishing pairwise linking numbers.  The $h$-function for links was introduced by Gorsky and N\'emethi \cite{Gor}. It is closely related to  $d$-invariants of large surgeries on links. For details, see Section \ref{background}.

We obtain lower bounds for the 4-genera of links in terms of the $h$-function. When the link has one component, we recover the lower bound for the slice genus given by Rasmussen. Here is our main result:

\begin{theorem}
\label{d-invariant inequality}
Let $\L=L_{1}\cup \cdots \cup L_{n}\subseteq S^{3}$ be an oriented link with vanishing pairwise linking numbers. Assume  that the link components $L_{i}$ bound pairwise disjoint, smoothly embedded surfaces $\Sigma_{i}\subseteq B^{4}$ of genera $g_{i}$. Then for any $\bm{v}=(v_{1}, \cdots, v_{n})\in \mathbb{Z}^{n}$, 
$$h(\bm{v})\leq \sum_{i=1}^{n} f_{g_{i}}(v_{i}).$$
where $h(\bm{v})$ is the $h$-function of $\L$, and $f_{g_{i}}: \Z\rightarrow \Z$ is defined as follows:
$$
f_{g_{i}}(v_{i}) = \left\{
        \begin{array}{ll}
           \left \lceil \dfrac{g_{i}-|v_{i}|}{2}\right\rceil & \quad |v_{i}|\leq g_{i} \\
            0 & \quad |v_{i}| > g_{i}
        \end{array}
    \right. 
$$
\end{theorem}

\begin{corollary}
For the link $\L$ in Theorem \ref{d-invariant inequality}, if $\bm{v}\succeq \bm{g}$, then
$h(\bm{v})=0$
where $\bm{g}=(g_{1}, \cdots, g_{n})$.

\end{corollary}
 
The proof of Theorem \ref{d-invariant inequality} is inspired by Rasmussen's argument for knots \cite{Ras}. We construct a non-positive definite Spin$^{c}$-cobordism from large surgeries on the link to the connected sum of circle bundles over closed, oriented surfaces of genera $g_{i}$.  Ozsv\'ath and Szab\'o  proved the  $d$-invariant inequality for a negative definite Spin$^{c}$-cobordism between rational homology spheres \cite{OSZ}. The $d$-invariant was generalized to \emph{standard} 3-manifolds, and 
Rasmussen proved the $d$-invariant inequality for a non-positive definite Spin$^{c}$-cobordism between standard 3-manifolds \cite{Lev, Ras} (see Subsection \ref{d-invariant}). We apply the result, and obtain the inequality between the $d$-invariants of large surgeries on the link and $d$-invariants of circle bundles. By using the $h$-function of the link to compute  $d$-invariants of large surgeries, we prove the inequality. 

\begin{theorem}
\label{unlink}
If $\L=L_{1}\cup \cdots \cup L_{n}\subset S^{3}$ is a (smoothly) slice $L$-space link, then $\L$ is an unlink. 
\end{theorem}

The idea of the proof goes as follows: the 4-genus for the slice link $\L$ is 0. By Theorem \ref{d-invariant inequality}, the $h$-function is identically 0. We compute the dual Thurston polytope of $\L$ by using the properties of $L$-space links and prove that $\L$ is an unlink. For details, see Subsection \ref{proof of unlink }.

As an application of the inequality in Theorem \ref{d-invariant inequality}, we can compare the following two sets. Let
$$\G(\L)=\lbrace \g= (g_{1}, g_{2}, \cdots, g_{n}) \mid g_{i}=g(\Sigma_{i}), \Sigma_{1}\sqcup \cdots \sqcup \Sigma_{n} \hookrightarrow B^{4}, \partial \Sigma_{i}=L_{i} \rbrace.$$
and
$$\G_{HF}(\L)=\lbrace \bm{v}=(v_{1}, \cdots, v_{n}) \mid h(\bm{v})=0  \textup{ and  } \bm{v}\succeq \bm{0} \rbrace.$$
The $4$-genus of the link $\L$  equals  $\min\limits_{\g\in \G(\L)}(g_{1}+\cdots+g_{n})$. By Theorem \ref{d-invariant inequality}, $\G(\L)\subseteq \G_{HF}(\L)$.

If $\L$ is an $L$-space link (see Definition \ref{definition of L-space link}), the $h$-function is explicitly determined by Alexander polynomials of the link and sublinks \cite[Section 3.3]{Boro}. We can describe the set $\G_{HF}(\L)$ in terms of Alexander polynomials explicitly (see Lemma \ref{determine Alexander polynomial}). Moreover, let $p_{i}, q_{i}$ be coprime positive integers where $1\leq i\leq n$, and let $L_{(p_{i}, q_{i})}$ denote the $(p_{i}, q_{i})$-cable of $L_{i}$. Then the link $\L_{cab}=L_{(p_{1}, q_{1})} \cup \cdots \cup L_{(p_{n}, 	q_{n})}$ is also an $L$-space link if $q_{i}/p_{i}$ is sufficiently large for each $1\leq i\leq n$, \cite[Proposition 2.8]{Boro}. For example, let $\L$ denote the 2-bridge link $b(4k^{2}+4k, -2k-1)$ which is an $L$-space link \cite{Liu}. Then for sufficiently large surgery coefficients, $\L_{cab}$ is also an $L$-space link, and $\G(\L_{cab})=\G_{HF}(\L_{cab})$ is shown as in Figure \ref{cable of links}. For details, see Section \ref{examples}. 

\begin{figure}[H]
\centering
\includegraphics[width=2.5in]{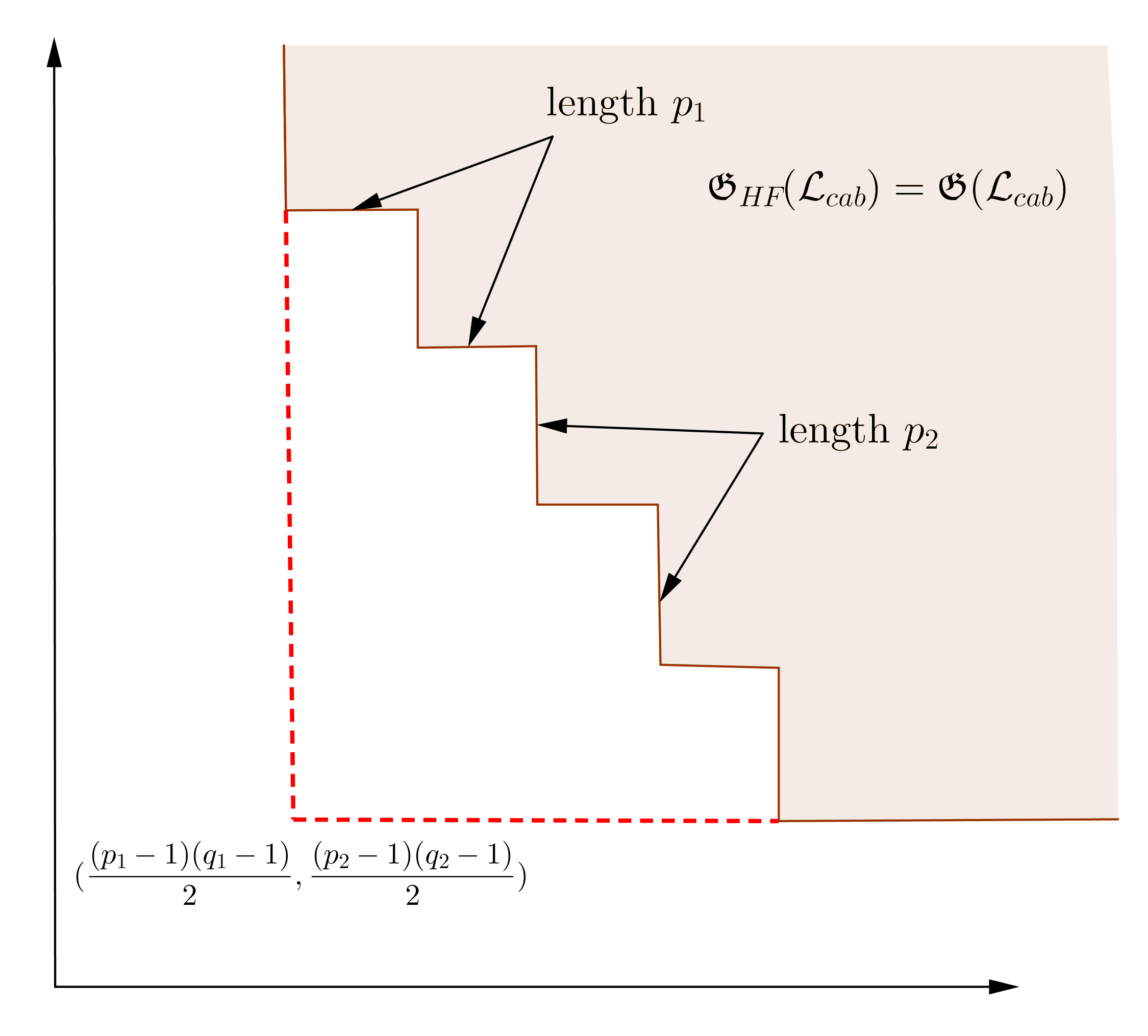}
\caption{The set $\G(\L_{cab})$ for the cable link  \label{cable of links}}

\end{figure} 

\begin{proposition}
\label{cable sharp}
If $\L\subset S^{3}$ is an $L$-space link such that $\G(\L)=\G_{HF}(\L)$, then for sufficiently large cables, $\L_{cab}$ also satisfies that $\G(\L_{cab})=\G_{HF}(\L_{cab})$. 

\end{proposition}

\medskip
{\bf Organization of the paper.} In  Subsection \ref{h-function general} and   Subsection \ref{d-invariant}, we review the definitions of   the $h$-function for links in $S^{3}$ and the $d$-invariants for standard 3-manifolds. In Subsection \ref{h-function}, we review the definition of $L$-space links, and the explicit formula to compute the $h$-function in terms of the Alexander polynomials of the link and sublinks. In Subsection \ref{heegaard floer homology}, we review the Heegaard Floer link homologies  of $L$-space links.   In Section \ref{main proof}, we prove Theorem \ref{d-invariant inequality}, Theorem \ref{unlink}, and give some lower bounds for the 4-genera of links. In Section \ref{examples}, we show some examples of $L$-space links including the 2-bridge links $b(4k^{2}+4k, -2k-1)$ where $k$ is some positive integer, and prove that $\G(\L)=\G_{HF}(\L)$ in these examples. Then the 4-genus is determined by  the Alexander polynomials. We also show the proof of Proposition \ref{cable sharp}.

\medskip
{\bf Notations and Conventions.} In this paper,  all the links are assumed to be oriented. We use $\L$ to denote a link in $S^{3}$, and  $L_{1}, \cdots, L_{n}$  to denote the link components. Then $\L_{1}$ and $\L_{2}$ denote different links in $S^{3}$, and $L_{1}$ and $L_{2}$ denote different components in the same link. We denote  vectors in the $n$-dimension lattice $\mathbb{Z}^{n}$ by bold letters. For two vectors $\bm{u}=(u_{1}, u_{2}, \cdots, u_{n})$ and $\bm{v}=(v_{1}, \cdots, v_{n})$ in $\mathbb{Z}^{n}$,  we write $\bm{u}\preceq \bm{v}$  if $u_{i}\leq v_{i}$ for each $1\leq i\leq n$, and $\bm{u}\prec \bm{v}$ if $\bm{u}\preceq \bm{v}$ and $\bm{u}\neq \bm{v}$.  Let $\bm{e}_{i}$ denote a vector in $\Z^{n}$ where the $i$th-entry is 1 and other entries are 0. For a subset $B\subset \lbrace 1, \cdots, n \rbrace$, let $\bm{e}_{B}=\sum_{i\in B} \bm{e}_{i}$. Similarly, we use $L_{B}\subset \L$ to denote the sublink $\cup_{i\in B} L_{i}$. Assume $\lbrace 1, \cdots, n\rbrace-B=\lbrace i_{1}, \cdots, i_{k} \rbrace$. 
For $\bm{y}=(y_{1}, \cdots, y_{n})\in \Z^{n}$, let $\bm{y}\setminus y_{B}=(y_{i_{1}}, \cdots, y_{i_{k}})$.  Let $\Delta_{\L}(t_{1}, \cdots, t_{n})$ denote the symmetrized Alexander polynomial of $\L$. Throughout this paper, we work over the field $\mathbb{F}=\mathbb{Z}/ 2\mathbb{Z}$.

\medskip
{\bf Acknowledgements.} I deeply appreciate Eugene Gorsky for introducing this interesting question to me and his patient guidance and helpful discussions during the project. I am also grateful to Allison Moore, Robert Lipschtiz, Jacob Rasmussen, Zhongtao Wu for useful discussions. The project is partially supported by NSF grant DMS-1700814.

\section{Background}
\label{background}

\subsection{The $h$-function}
\label{h-function general}
Ozsv\'ath and Szab\'o associated chain complexes $CF^{-}(M), \widehat{CF}(M),$ $ CF^{\infty}(M)$ and $CF^{+}(M)$ to an admissible  Heegaard diagram for a closed oriented connected 3-manifold $M$ \cite{SZ2}. The homologies of these chain complexes are called Heegaard Floer homologies $HF^{-}(M), \widehat{HF}(M), HF^{\infty}(M)$ and $HF^{+}(M)$, which are 3-manifold invariants. A  link $\L=L_{1}\cup \cdots \cup L_{n}$ in  $M$ defines a filtration on the link Floer complex $CF^{-}(M)$ \cite{Cip, SZ3}. For links in $S^{3}$, the filtration is indexed by an $n$-dimensional lattice $\mathbb{H}$ which is defined as follows: 

\begin{definition}
For an oriented link $\L=L_{1}\cup \cdots \cup L_{n}\subset S^{3}$, define $\mathbb{H}(\L)$ to be an affine lattice over $\Z^{n}$:
$$\H(\L)=\oplus_{i=1}^{n}\H_{i}(\L), \quad \H_{i}(\L)=\Z+\dfrac{lk(L_{i}, \L\setminus L_{i})}{2}$$
where $lk(L_{i}, \L\setminus L_{i})$ denotes the linking number of $L_{i}$ and $\L\setminus L_{i}$. 
\end{definition} 

Given $\bm{s}=(s_{1}, \cdots, s_{n})\in \H(\L)$,  the \emph{generalized Heegaard Floer complex} $A^{-}(\L, s)$ is defined to be a subcomplex of $CF^{-}(S^{3})$ corresponding to the filtration indexed by $\bm{s}$ \cite{Cip}. For $\bm{v}\preceq \bm{s}$, $A^{-}(\L, \bm{v})\subseteq A^{-}(\L, \bm{s})$. The link homology $HFL^{-}$ is defined as the homology of the associated graded complex:
\begin{equation}
\label{link Floer}
HFL^{-}(\L, \bm{s})=H_{\ast}\left(  A^{-}(\L, \bm{s})/ \sum_{\bm{v}\prec \bm{s}}A^{-}(\L, \bm{v})\right).
\end{equation}
The complex $A^{-}(\L, \bm{s})$ is a finitely generated module over the polynomial ring $\F[U_{1},\cdots, U_{n}]$ where the action of  $U_{i}$ drops the homological grading by 2 and drops the $i$-th filtration $A_{i}$ by 1  \cite{SZ3}. Hence, $U_{i}A^{-}(\L, \bm{s})\subseteq A^{-}(\L, \bm{s}-\bm{e}_{i})$. All the actions $U_{i}$ are homotopic to each other on each $A^{-}(\L, \bm{s})$, and the homology of $A^{-}(\L, \bm{s})$ can be regarded as a $\F[U]$-module where $U$ acts as $U_{1}$ \cite{Gor, SZ3}.

By the large surgery theorem \cite{Cip} , the homology of $A^{-}(\L, \bm{s})$ is isomorphic to the Heegaard Floer homology  of a large surgery on the link $\L$ equipped with some Spin$^{c}$-structure as a $\F[U]$-module \cite{Cip}. Then the homology of $A^{-}(\L, \bm{s})$ consists of the free part which is a direct sum of one copy of $\F[U]$ and some $U$-torsion. 

\begin{definition}\cite[Definition 3.9]{Boro}
For an oriented link $\L\subseteq S^{3}$, we define the $H$-function $H_{\L}(\bm{s})$ by saying that $-2H_{\L}(\bm{s})$ is the maximal homological degree of the free part of $H_{\ast}(A^{-}(\L, \bm{s}))$ where $\bm{s}\in \H$. 

\end{definition}

\begin{lemma}\cite[Proposition 3.10]{Boro}
\label{h-function increase}
For an oriented link $\L\subseteq S^{3}$, the $H$-function $H_{\L}(\bm{s})$ takes nonnegative values, and $H_{\L}(\bm{s}-\bm{e}_{i})=H_{\L}(\bm{s})$ or $H_{\L}(\bm{s}-\bm{e}_{i})=H_{\L}(\bm{s})+1$ where $\bm{s}\in \H$. 

\end{lemma}

For an $n$-component link $\L$ with vanishing pairwise linking numbers, $\H(\L)=\Z^{n}$. The $h$-function $h_{\L}(\bm{s})$  is defined as 
$$h_{\L}(\bm{s})=H_{\L}(\bm{s})-H_{O}(\bm{s})$$
where $O$ denotes the unlink with $n$ components, and $\bm{s}\in \Z^{n}$. Recall that for split links $\L$, the $H$-function 
$H(\L, \bm{s})=H_{L_{1}}(s_{1})+\cdots +H_{L_{n}}(s_{n})$ where $H_{L_{i}}(s_{i})$ is the $H$-function of the link component $L_{i}$, \cite[Proposition 3.11]{Boro}.  Then $H_{O}(\bm{s})=H(s_{1})+\cdots H(s_{n})$ where $H(s_{i})$ denotes the $H$-function of the unknot. More precisely, $H_{O}(\bm{s})=\sum_{i=1}^{n}(|s_{i}|-s_{i})/2$.  Then $H_{\L}(\bm{s})=h_{\L}(\bm{s})$ for all $\bm{s}\succeq \bm{0}$.

For the rest of this subsection, we use $\L=L_{1}\cup \cdots \cup L_{n}\subset S^{3}$ to denote links with vanishing pairwise linking numbers.  Consider the set 
$$\G_{HF}(\L)=\lbrace \bm{s}=(s_{1}, \cdots, s_{n})\in \Z^{n} \mid h(\bm{s}) =0, \bm{s}\succeq \bm{0} \rbrace.$$
We obtain the following properties of the set $\G_{HF}(\L)$. 

\begin{lemma}
\label{shape 1}
If $\bm{x}\in \G_{HF}(\L)$ and $\bm{y}\succeq \bm{x}$, then $\bm{y}\in \G_{HF}(\L)$. Equivalently, if $\bm{x}\notin \G_{HF}(\L)$ and $\bm{y}\preceq \bm{x}$, then $\bm{y}\notin \G_{HF}(\L)$. 

\end{lemma}

\begin{proof}
This is straightforward from Lemma \ref{h-function increase}. 
\end{proof}

\begin{lemma}
\label{shape 2}
If $\bm{s}=(s_{1}, \cdots, s_{n})\in \G_{HF}(\L)$, then $\bm{s}\setminus s_{i}\in \G_{HF}(\L\setminus L_{i})$ for all $1\leq i\leq n$. Moreover, if $\bm{s}\setminus s_{i}\in \G_{HF}(\L\setminus L_{i})$, then for $s_{i}$ sufficiently large, $\bm{s}=(s_{1}, \cdots, s_{n})\in \G_{HF}(\L)$.

\end{lemma}

\begin{proof}
For an oriented link $\L$, there exists a natural forgetful map $\pi_{i}: \H(\L)\rightarrow \H(\L\setminus L_{i})$ \cite{Cip}. If $\L$ has vanishing pairwise linking numbers, $\pi_{i}(\bm{s})=\bm{s}\setminus s_{i}$ where $\bm{s}\in \Z^{n}$. Suppose that $\bm{s}\in \G_{HF}(\L)$. Then $h_{\L}(\bm{s})=H_{\L}(\bm{s})=0$. By Lemma \ref{h-function increase}, $H_{\L}(\bm{s}+t\bm{e}_{i})=0$ for all $i$ and $t>0$. Recall that  $H_{\L}(\bm{s}+t\bm{e}_{i})=H_{\L\setminus L_{i}}(\bm{s}\setminus s_{i})$ for sufficiently large $t$, \cite[Proposition 3.12]{Boro}. Then $H_{\L\setminus L_{i}}(\bm{s}\setminus s_{i})=0$. Thus, $\bm{s}\setminus s_{i}\in \G_{HF}(\L\setminus L_{i})$. 

Conversely, if $\bm{s}\setminus s_{i}\in \G_{HF}(\L\setminus L_{i})$, and $s_{i}$ is sufficiently large, then $H_{\L}(\bm{s})=H_{\L\setminus L_{i}}(\bm{s}\setminus s_{i})=0$, which implies that $\bm{s}\in \G_{HF}(\L)$.

\end{proof}

\begin{definition}
\label{maximal lattice point}
A lattice point $\bm{s}\in \Z^{n}$ is \emph{maximal} if $\bm{s}\notin \G_{HF}(\L)$, but $\bm{s}+\bm{e}_{i}\in \G_{HF}(\L)$ for all $1\leq i \leq n$. 

\end{definition}

\begin{lemma}
\label{shape 3}
The set $\G_{HF}(\L)$ is determined by the set of maximal lattice points and $\G_{HF}(\L\setminus L_{i})$ for all $1\leq i\leq n$.

\end{lemma}

\begin{proof}
We claim that  $\bm{x}\notin \G_{HF}(L)$ if and only if either $\bm{x}\preceq \bm{z}$ for some maximal lattice point $\bm{z}\in \Z^{n}$ or $\bm{x}\setminus x_{i}\notin \G_{HF}(\L\setminus L_{i})$ for  some $i$ where $\bm{x}=(x_{1}, \cdots, x_{n})$. For the ``if" part, assume that $\bm{x}\in \G_{HF}(\L)$. Then $\bm{z}\in \G_{HF}(\L)$ if $\bm{z}\succeq \bm{x}$ and $\bm{x}\setminus x_{i}\in \G_{HF}(\L\setminus L_{i})$ for all $i$ by Lemma \ref{shape 1} and Lemma \ref{shape 2}, which contradicts to the assumption. 

For the ``only if" part, assume that $\bm{x}\notin \G_{HF}(\L)$ and $\bm{x}\setminus x_{i}\in \G_{HF}(\L\setminus L_{i})$ for all $i$. It suffices to find a maximal lattice point $\bm{z}$ such that $\bm{x}\preceq \bm{z}$. If $H_{\L}(\bm{x}+\bm{e}_{i})=0$ for all $i$, we let $\bm{z}=\bm{x}$. Otherwise, suppose $H_{\L}(\bm{x}+\bm{e}_{i})\neq 0$ for some $1\leq i \leq n$. There exists some constant $t_{i}$ such that $H_{\L}(\bm{x}+t_{i}\bm{e}_{i})\neq 0$, and $H_{\L}(\bm{x}+(t_{i}+1)\bm{e}_{i})=0$ since $\bm{x}\setminus x_{i}\in \G_{HF}(\L\setminus L_{i})$. If for all $j\neq i$, $H_{\L}(\bm{x}+t_{i}\bm{e}_{i}+\bm{e}_{j})=0$, we let $\bm{z}=\bm{x}+t_{i}\bm{e}_{i}$. Otherwise, we repeat this process. The process stops after finite steps. Thus, there exists a maximal lattice point $\bm{z}$ such that $\bm{x}\preceq \bm{z}$.

\end{proof}

\subsection{The $d$-invariant}
\label{d-invariant}
For a rational homology sphere $M$  with a Spin$^{c}$-structure $\mathfrak{s}$, the Heegaard Floer homology $HF^{\infty}(M, \mathfrak{s})\cong \F[U, U^{-1}]$ and  $HF^{+}(M, \mathfrak{s})$ is absolutely graded where  the free part is isomorphic to $\F[U^{-1}]$. Define the $d$-invariant of $(M, \mathfrak{s})$ to be the absolute grading of $1\in \mathbb{F}[U^{-1}]$ \cite{OSZ}. Equivalently,  the $d$-invariant of $(M, \mathfrak{s})$ is the maximum grading of $x\in HF^{-}(M, \mathfrak{s})$, which has a nontrivial image in $HF^{\infty}(M, \mathfrak{s})$. 

We define \emph{standard} 3-manifolds following \cite[Section 9]{OSZ}: 

\begin{definition}
A closed, oriented $3$-manifold $M$ is \emph{standard} if for each torsion Spin$^{c}$-structure $\mathfrak{s}$, 
$$HF^{\infty}(M, \mathfrak{s})\cong (\wedge^{b}H^{1}(M, \mathbb{F}))\otimes_{\mathbb{F}} \mathbb{F}[U, U^{-1}],$$
where $b=b_{1}(M)$. 

\end{definition}

\begin{remark}
If $M$ is standard, then  rk$HF^{\infty}(M, \mathfrak{s})=2^{b}$ as a $\F[U, U^{-1}]$-module.  
\end{remark}

Let $M_{1}$ and $M_{2}$ be a pair of oriented closed 3-manifolds equipped with Spin$^{c}$-structures $\mathfrak{s}_{1}$ and $\mathfrak{s}_{2}$ respectively. There is a connected sum formula for the Heegaard Floer homology \cite[Theorem 6.2]{SZ2}:
$$HF^{\infty}(M_{1}\# M_{2}, \mathfrak{s}_{1}\# \mathfrak{s}_{2})\cong H_{\ast}(CF^{\infty}(M_{1}, \mathfrak{s}_{1})\otimes_{\F[U, U^{-1}]} CF^{\infty}(M_{2}, \mathfrak{s}_{2})).$$
By the algebraic K\"{u}nneth theorem, if $M_{1}$ and $M_{2}$ are standard, then $M_{1}\# M_{2}$ is also standard.

If a 3-manifold  $M$ has a positive first Betti number (i.e $b_{1}(M)>0$), the exterior algebra $\Lambda^{\ast}(H_{1}(M; \F))$ acts on the homology groups $HF^{\infty}(M, \mathfrak{s}), HF^{+}(M, \mathfrak{s}), HF^{-}(M, \mathfrak{s})$ and $\widehat{HF}(M, \mathfrak{s})$, \cite[Section 4.2.5]{SZ2}. Define the subgroup $\mathfrak{A}_{\mathfrak{s}}\subset HF^{\infty}(M, \mathfrak{s})$ as follows:
$$\mathfrak{A}_{\mathfrak{s}}=\lbrace x\in HF^{\infty}(M, \mathfrak{s})\mid \gamma \cdot x=0 \quad \forall \gamma\in H_{1}(M, \F) \rbrace. $$
If $M$ is standard, $\mathfrak{A}_{\mathfrak{s}}\cong \mathbb{F}[U, U^{-1}]$, and its image under the map $\pi: HF^{\infty}(M, \mathfrak{s})\rightarrow HF^{+}(M, \mathfrak{s})$ is isomorphic to $\mathbb{F}[U^{-1}]$.

\begin{definition}
For a standard 3-manifold $M$ equipped with a torsion Spin$^{c}$-structure $\mathfrak{s}$, the $d$-invariant $d(M, \mathfrak{s})$ is defined as the absolute grading of $1\in \pi(\mathfrak{A}_{\mathfrak{s}})\cong \mathbb{F}[U^{-1}].$

\end{definition}

Given rational homology spheres $Y_{1}, Y_{2}$ and a negative definite Spin$^{c}$-cobordism $(W, \mathfrak{t})$ from $(Y_{1}, \mathfrak{t}_{1})$ to $(Y_{2}, \mathfrak{t}_{2})$, the $d$-invariant of $Y_{2}$ is no smaller than the $d$-invariant of  $Y_{1}$  up to the degree shift of the cobordism \cite{OSZ}. Rasmussen generalized this inequality for non-positive definite Spin$^{c}$-cobordisms from rational homology spheres to standard 3-manifolds \cite{Ras}. 

\begin{proposition}\cite{Ras}
\label{standard}
Suppose that  $W$ is a connected cobordism from closed, oriented and connected 3-manifolds $Y_{1}$ to $Y_{2}$ such that 
$b_{1}(Y_{1})=b_{1}(W)=b_{2}^{+}(W)=0,$
and $Y_{2}$ is  standard. 
Let $\mathfrak{s}$ be a Spin$^{c}$-structure on $W$ whose restriction $\mathfrak{s}_{i}$ to $Y_{i}$ is torsion. Then 
$$F^{\infty}_{W, \mathfrak{s}}: HF^{\infty}(Y_{1}, \mathfrak{s}_{1})\rightarrow HF^{\infty}(Y_{2}, \mathfrak{s}_{2})$$
maps $HF^{\infty}(Y_{1}, \mathfrak{s}_{1})$ isomorphically onto $\mathfrak{A}_{\mathfrak{s}_{2}}\subset HF^{\infty}(Y_{2}, \mathfrak{s}_{2})$.
Moreover, we have
$$d(Y_{1}, \mathfrak{s}_{1})+\deg F^{+}_{W, \mathfrak{s}}\leq d(Y_{2}, \mathfrak{s}_{2}) $$
where the degree  of $F^{+}_{W, \mathfrak{s}}: HF^{+}(Y_{1}, \mathfrak{s}_{1})\rightarrow HF^{+}(Y_{2}, \mathfrak{s}_{2})$ is 
$\dfrac{1}{4}(c_{1}^{2}(\mathfrak{s})-3\sigma(W)-2\chi (W)).$

\end{proposition}

The $d$-invariants of large surgeries on a link $\L=L_{1}\cup \cdots \cup L_{n}\subset S^{3}$ can be computed in terms of  the $H$-function of the link by the large surgery theorem \cite{Cip}. Choose a framing vector $\bm{q}=(q_{1}, \cdots, q_{n})$ where $q_{1}, \cdots, q_{n}$ are sufficiently large. Let $\Lambda$ denote the linking matrix where $\Lambda_{ij}$ is the linking number of $L_{i}$ and $L_{j}$ when $i\neq j$, and $\Lambda_{ii}=q_{i}$. 

Attach $n$ 2-handles to the $4$-ball $B^{4}$ along $L_{1}, L_{2}\cdots, L_{n}$ with framings $q_{1}, \cdots, q_{n}$. We obtain a 2-handlebody $W$ with boundary $\partial W=S^{3}_{\bm{q}}(\L)$ which is the $3$-manifold obtained by doing surgery along $L_{1}, L_{2}, \cdots, L_{n}$ with surgery coefficients $q_{1}, \cdots, q_{n}$ respectively. Assume  that $\det (\Lambda)\neq 0$, then $S^{3}_{\bm{q}}(\L)$ is a rational homology sphere with $|H_{1}(S^{3}_{\bm{q}}(\L))| =| \textup{det}(\Lambda)|$. Note that if $\L$ has vanishing pairwise linking numbers, then $\Lambda$ is a diagonal matrix with $\Lambda_{ii}=q_{i}$, and $\det(\Lambda)=q_{1}\cdots q_{n}\neq 0$.  The Spin$^{c}$-structures on $S^{3}_{\bm{q}}(\L)$ are enumerated  as follows:

\begin{lemma}\cite[Section 9.3]{Cip}
There are natural identifications: 
$$H^{2}(S_{\bm{q}}^{3}(\L))\cong H_{1}(S_{\bm{q}}^{3}(\L))\cong \mathbb{Z}^{n} / \mathbb{Z}^{n}\Lambda $$
such that $c_{1}(\mathfrak{s})=[2\mathfrak{s}]$ for any $\mathfrak{s}\in \textup{Spin}^{c}(S_{\bm{q}}^{3}(\L))\cong \mathbb{H}(\L)/ \mathbb{Z}^{n}\Lambda $. 

\end{lemma}

Fix $\zeta =(\zeta_{1}, \cdots, \zeta_{n})$, a small real vector whose entries are linearly independent over $\mathbb{Q}$. Let $P(\Lambda)$ be the hyper-parallelepiped with vertices
$$\zeta+\dfrac{1}{2}(\pm \Lambda_{1}, \pm \Lambda_{2}, \cdots, \pm \Lambda_{n}), $$
where all combinations of the signs are used and $\Lambda_{1}, \cdots, \Lambda_{n}$ are column vectors of the matrix $\Lambda$. Denote 
$$P_{\mathbb{H}}(\Lambda)=P(\Lambda)\cap \mathbb{H}(\L),$$
where $\mathbb{H}(\L)$ is the lattice for $\L$. 

\begin{proposition}\cite[Section 10.1]{Cip}
\label{spin extension}
For any $\bm{v}\in P_{\mathbb{H}}(\Lambda)$ there exists a unique Spin$^{c}$-structure $\mathfrak{s}_{\bm{v}}$ on $S_{\bm{q}}(\L)$ which extends to a Spin$^{c}$-structure $\mathfrak{t}_{\bm{v}}$ on $W$  with $c_{1}(\mathfrak{t}_{\bm{v}})=2\bm{v}-(\Lambda_{1}+\cdots+\Lambda_{n})$.

\end{proposition}

Remove a ball $B^{4}$ from the 2-handlebody $W$. We obtain a Spin$^{c}$-cobordism $\mathcal{U}$ from $(S^{3}, \mathfrak{s}_{0})$ to $(S^{3}_{\bm{q}}(\L), \mathfrak{s}_{\bm{v}})$. By reversing the orientation of $\mathcal{U}$, we obtain a Spin$^{c}$-cobordism $\mathcal{U}'$ equipped with the Spin$^{c}$-structure $\mathfrak{t}_{\bm{v}}$ from $(S^{3}_{\bm{q}}(\L), \mathfrak{s}_{\bm{v}})$ to $(S^{3}, \mathfrak{s}_{0})$.

\begin{theorem}\cite{Boro, Cip}
\label{d-invariant formula}
For $\bm{v}\in P_{\H}(\Lambda)$, the $d$-invariant of a large surgery with surgery coefficients $\bm{q}$ on $\L$ is given by 
$$d(S^{3}_{\bm{q}}(\L), \mathfrak{s}_{\bm{v}})=-\textup{deg} F_{(\mathcal{U}', \mathfrak{t}_{\bm{v}})}-2H(\bm{v}),$$
where  $\deg F_{\mathcal{U}', \mathfrak{t}_{\bm{v}}}$ is the grading shift of the cobordism $\mathcal{U}'$ with Spin$^{c}$-structure $\mathfrak{t}_{\bm{v}}$. It does not depend on the link, but depends on the linking matrix $\Lambda$. 
\end{theorem}

\subsection{The $h$-function of $L$-space links}
\label{h-function}

In \cite{SZ4}, Ozsv\'ath and Szab\'o introduced the concept of $L$-spaces. 

\begin{definition}
A 3-manifold $M$ is an $L$-space if it is a rational homology sphere and its Heegaard Floer homology has minimal possible rank: for any Spin$^{c}$-structure $\mathfrak{s}$, $\widehat{HF}(M, \mathfrak{s})=\F$, and $HF^{-}(Y, \mathfrak{s})$ is a free $\F[U]$-module of rank 1. 
\end{definition}

In terms of the large surgery, Gorsky and N\'emethi defined $L$-space links in \cite{Gor}.

\begin{definition}
\label{definition of L-space link}
An oriented $n$-component link $\L\subset S^{3}$ is an $L$-space link if there exists  $\bm{0}\prec \bm{p}\in \mathbb{Z}^{n}$ such that the surgery manifold $S^{3}_{\bm{q}}$ is an $L$-space for any $\bm{q}\succeq \bm{p}$. 

\end{definition}

For $L$-space links $\L$, $H_{\ast}(A^{-}(\L, \bm{s}))=\F[[U]]$ \cite{Liu}. By Equation \eqref{link Floer} and the inclusion-exclusion formula, one can write \cite{Boro}:
\begin{equation}
\label{computation of h-function 1}
\chi(HFL^{-}(\L, \bm{s}))=\sum_{B\subset \lbrace 1, \cdots, n \rbrace}(-1)^{|B|-1}H_{\L}(\bm{s}-\bm{e}_{B}). 
\end{equation}

The Euler characteristic $\chi(HFL^{-}(\L, \bm{s}))$ was computed in \cite{SZ3}:
\begin{equation}
\label{computation 3}
\tilde{\Delta}(t_{1}, \cdots, t_{n}):=\sum_{\bm{s}\in \H(\L)}\chi(HFL^{-}(\L, \bm{s}))t_{1}^{s_{1}}\cdots t_{n}^{s_{n}}
\end{equation}
where $\bm{s}=(s_{1}, \cdots, s_{n})$, and
$$
\tilde{\Delta}_{\L}(t_{1}, \cdots, t_{n}): = \left\{
        \begin{array}{ll}
           (t_{1}\cdots t_{n})^{1/2} \Delta_{\L}(t_{1}, \cdots, t_{n}) & \quad \textup{if } n >1, \\
            \Delta_{\L}(t)/(1-t^{-1}) & \quad  \textup{if } n=1. 
        \end{array}
    \right. 
$$

\begin{theorem}\cite{Gor}
\label{computation 4}
The $H$-function of an $L$-space link is determined by the Alexander polynomials of its sublinks as follows:
\begin{equation}
\label{converse computation of h-function}
H_{\L}(\bm{s})=\sum_{\L'\subset \L}(-1)^{\#\L'-1}\sum_{\bm{u'}\succeq \pi_{\L'}(\bm{s}+\bm{1})}\chi(HFL^{-}(\L', \bm{u'})),
\end{equation}
where $\bm{1}=(1, \cdots, 1)$. 
\end{theorem}
\begin{remark}
For $L$-space links with two components, the explicit formula for the $H$-function can also be found in \cite{Liu}. 
\end{remark}

Consider  $L$-space links  $\L$ with vanishing pairwise linking numbers. The set $\G_{HF}(\L)$ can also be described in terms of Alexander polynomials of the link and sublinks. 

\begin{lemma}
\label{determine Alexander polynomial}
For an  $n$-component $L$-space link $\L\subseteq S^{3}$ with vanishing pairwise linking numbers, $\bm{s}\in \G_{HF}(\L)$ if and only if for all $\bm{y}=(y_{1}, \cdots, y_{n})\succ \bm{s}$, the coefficients of $t_{1}^{y_{1}}\cdots t_{n}^{y_{n}}$ in $\tilde{\Delta}_{\L}(t_{1}, \cdots, t_{n})$ are 0, and the coefficients  corresponding to $\bm{y}\setminus y_{B}$ in $\tilde{\Delta}_{\L\setminus L_{B}}(t_{i_{1}}, \cdots, t_{i_{k}})$ are also 0 for all $B\subset\lbrace 1, \cdots, n\rbrace$. 
\end{lemma}

\begin{proof}
For the ``if " part, observe that $\chi(HFL^{-}(\L, \bm{y}))=0$  and $\chi(HFL^{-}(\L\setminus L_{B}), \bm{y}\setminus y_{B})=0$ for all $\bm{y}\succ \bm{s}$ and $B\subset \lbrace 1, \cdots, n \rbrace$ by Equation \eqref{computation 3}. Then $H_{\L}(\bm{s})=0$ by Theorem \ref{computation 4}, and $\bm{s}\in \G_{HF}(\L)$.  For the ``only if " part, suppose that $\bm{s}\notin \G_{HF}(\L)$.  By Lemma \ref{shape 3}, either there exists a maximal vector $\bm{z}\notin \G_{HF}(\L)$ such that $\bm{s}\preceq \bm{z}$ or there exists some $1\leq j \leq n$, such that $\bm{s}\setminus s_{j}\notin \G_{HF}(\L\setminus L_{j})$. We claim that for all maximal lattice points $\bm{z}$, $\chi(HFL^{-}(\L, \bm{z}+\bm{1}))\neq 0$. Since $\bm{z}$ is maximal, $h_{\L}(\bm{z})=1$, and for any subset $B\subset \lbrace 1, \cdots, n\rbrace$, $h_{\L}(\bm{z}+\bm{e}_{B})=0$. By Equation \eqref{computation of h-function 1}, $\chi(HFL^{-}(\L, \bm{z}+\bm{1}))=(-1)^{n}\neq 0$. If $\bm{s}\preceq \bm{z}$, the coefficient of $\bm{z}+\bm{1}\succ \bm{s}$ in $\tilde{\Delta}_{\L}(t_{1}, \cdots , t_{n})$ equals 0 which contradicts to our assumption. If $\bm{s}\setminus s_{i}\notin \G_{HF}(\L\setminus L_{i})$,  we use the induction to get a contradiction.

\end{proof}

\subsection{The Heegaard Floer link homology}
\label{heegaard floer homology}

 Ozsv\'ath and Szab\'o associated the multi-graded link invariants $HFL^{-}(\L)$ and $\widehat{HFL}(\L)$ to links $\L\subset S^{3}$ where $HFL^{-}(L)$ was defined in Equation \eqref{link Floer}, and $\widehat{HFL}(L)$ is defined as follows \cite{ Hom, SZ3}:
$$\widehat{HFL}(\L, \bm{s})=H_{\ast}\left(A^{-}(\L, \bm{s})/\left[\sum_{i=1}^{n}A^{-}(\bm{s}-\bm{e}_{i})\oplus \sum_{i=1}^{n}U_{i}A^{-}(\bm{s}+\bm{e}_{i})\right]\right).$$

If $\L$ is an $L$-space link, there exist spectral sequences converging to $HFL^{-}(\L)$ and $\widehat{HFL}(\L)$ respectively \cite{Hom, Gor}. 

\begin{proposition}\cite[Theorem 1.5.1]{Gor}
\label{spectral sequence 1}
For an oriented $L$-space link $\L\subset  S^{3}$ with n components and $\bm{s}\in \H(\L)$, there exists a spectral sequence with $E_{\infty}=HFL^{-}(\L, \bm{s})$ and 
$$E_{1}=\bigoplus_{B\subset \lbrace 1, \cdots, n \rbrace} H_{\ast}(A^{-}(\L, \bm{s}-\bm{e}_{B})), $$
where the differential in $E_{1}$ is induced by inclusions.

\end{proposition}
\begin{remark}
Precisely, the differential $\partial_{1}$ in the $E_{1}$-page is 
$$\partial_{1}(z(\bm{s}-\bm{e}_{B}))=\sum_{i\in B}U^{H(\bm{s}-\bm{e}_{B})-H(\bm{s}-\bm{e}_{B}+\bm{e}_{i})}z(\bm{s}-\bm{e}_{B}+\bm{e}_{i}),$$
where $z(\bm{s}-\bm{e}_{B})$ denotes the unique generator in $H_{\ast}(A^{-}(\L, \bm{s}-\bm{e}_{B}))$ with the homological grading $-2H(\bm{s}-\bm{e}_{B})$. 

\end{remark}

\begin{proposition}\cite[Proposition 3.8]{Hom}
\label{spectral sequence 2}
For an $L$-space link $\L\subset S^{3}$ with n components and $\bm{s}\in \H(\L)$, there exists a spectral sequence converging to $\hat{E}_{\infty}=\widehat{HFL}(\L, \bm{s})$ with $E_{1}$ page
$$\hat{E}_{1}=\bigoplus_{B\subset \lbrace 1, \cdots, n \rbrace}  HFL^{-}(\L, \bm{s}+\bm{e}_{B}).$$

\end{proposition}

There is a nice symmetric property of $\widehat{HFL}(\L)$ proved by Oszv\'ath and Szab\'o \cite{SZ5}:
\begin{equation}
\label{symmetry}
\widehat{HFL}_{\ast}(\L, \bm{s})\cong \widehat{HFL}_{\ast}(\L, -\bm{s}).
\end{equation}

\section{The proof of the Main theorem}
\label{main proof}

\subsection{The Spin$^{c}$-cobordism}
In this section, we use $\L=L_{1}\cup \cdots \cup L_{n} \subset S^{3}$  to denote an oriented link  with vanishing pairwise linking numbers. Suppose that link components $L_{i}$ bound pairwise disjoint smoothly embedded surfaces  $\Sigma_{i}$ of genera $g_{i}$ in $B^{4}$ for all $1\leq i\leq n$. Attach $n$ 2-handles to the $4$-ball $B^{4}$ along $L_{1}, L_{2}\cdots, L_{n}$ with framings $-p_{1}, -p_{2}, \cdots, -p_{n}$. We obtain a 2-handlebody $W$ with boundary $\partial W=S^{3}_{-p_{1},\cdots, -p_{n}}(\L)$ which is the $3$-manifold obtained by doing surgery along $L_{1}, L_{2}, \cdots, L_{n}$ with surgery coefficients $-p_{1}, -p_{2}, \cdots, -p_{n}$ respectively. The linking matrix $\Lambda$ is a diagonal matrix with $\lambda_{ii}=-p_{i}$. Observe that $\det (\Lambda)\neq 0$, then $S^{3}_{-p_{1}, \cdots, -p_{n}}(\L)$ is a rational homology sphere. 

Let $\Sigma'_{i}$ be the closed surface in $W$ which is the union of $\Sigma_{i}$ and the core of the 2-handle attached along $L_{i}$. Then $\Sigma'_{i}$ are also pairwise disjoint. Observe that $W$ is homotopy equivalent to the wedge of $n$ copies of $S^{2}$. Thus, 
$H_{2}(W)=\Z^{n}$  
and $[\Sigma'_{i}]$ are generators of  $H_{2}(W)$.  The self intersection number of each $\Sigma'_{i}$ in $W$ is $-p_{i}$. 

Take   small tubular neighborhoods $nd(\Sigma'_{i})$ of $\Sigma'_{i}$ such that they are also pairwise disjoint. Then $nd(\Sigma'_{i})$ is a disk bundle over $\Sigma'_{i}$
and its boundary $\partial(nd(\Sigma'_{i}))$ is a circle bundle $B_{-p_{i}}$ with Euler number $-p_{i}$. The boundary connected sum $\mathfrak{D}$ of the disk bundles over $\Sigma'_{i}$ in $W$ is obtained by identifying  smoothly embedded balls $B^{3}_{i}\subset B_{-p_{i}}$ and $B^{3}_{i+1}\subset B_{-p_{i+1}}$ for $1\leq i\leq n-1$, and $\mathfrak{D}$ is also a smooth oriented manifold \cite[Section 6.3]{Ant}. Observe that $\mathfrak{D}$ has the homotopy type of $D_{-p_{1}}\vee \cdots \vee D_{-p_{n}}$ where $D_{-p_{i}}$ denotes the disk bundle over $\Sigma'_{i}$. Since $D_{-p_{i}}$ is homotopy equivalent to $\Sigma'_{i}$, then 
$$\widetilde{H}_{i}(\mathfrak{D})\cong \bigoplus_{i=1}^{n}\widetilde{H_{i}}(\Sigma'_{i}).$$

Let $X$ denote the complement of $\mathfrak{D}$ in $W$.  It is a cobordism from   $B_{-p_{1}}\# \cdots \# B_{-p_{n}}$ to $S^{3}_{-p_{1}, \cdots, -p_{n}}(\L)$. Let $\bar{X}$ be the cobordism from $S^{3}_{-p_{1}, \cdots, -p_{n}}(\L)$ to $ B_{-p_{1}}\# \cdots \# B_{-p_{n}}$ obtained by reversing the orientation of $X$. 

\begin{proposition}
\label{circle bundle cohomology}
For a circle bundle $B_{-m}$ over a closed oriented surface of genus $g$ and Euler number $-m<0$, its cohomology is the following:
$$H^{1}(B_{-m})\cong \mathbb{Z}^{2g}, \quad H^{2}(B_{-m})\cong \mathbb{Z}^{2g}\oplus \mathbb{Z}_{m}, \quad H^{3}(B_{-m})\cong \mathbb{Z}.$$

\end{proposition}

\begin{proof}

For the circle bundle $B_{-m}$,   we have the following long exact sequence by using the  Leray sequence: 
$$0\rightarrow H^{1}(\Sigma_{g})\rightarrow H^{1}(B_{-m})\rightarrow H^{0}(\Sigma_{g})\xrightarrow {\wedge e} H^{2}(\Sigma_{g})\rightarrow H^{2}(B_{-m})\rightarrow H^{1}(\Sigma_{g})\rightarrow 0$$
where $e$ is the Euler class. Then we compute that
$$0\rightarrow \mathbb{Z}^{2g} \rightarrow H^{1}(B_{-m}) \rightarrow \mathbb{Z} \xrightarrow {\cdot m} \mathbb{Z}\rightarrow H^{2}(B_{-m})\rightarrow \mathbb{Z}^{2g}\rightarrow 0.$$
Thus,  $H^{1}(B_{-m})\cong \mathbb{Z}^{2g}$ and we have the following short exact sequence:
$$0\rightarrow \mathbb{Z}_{m}\rightarrow H^{2}(B_{-m})\rightarrow \mathbb{Z}^{2g}\rightarrow 0. $$
Since $ \mathbb{Z}^{2g}$ is free, the exact sequence splits and $H^{2}(B_{-m})\cong \mathbb{Z}^{2g}\oplus \mathbb{Z}_{m}$.  The circle bundle $B_{-m}$ is oriented and closed, so $H^{3}(B_{-m})\cong \mathbb{Z}$. 

\end{proof}

\begin{lemma}
 \label{connect sum homology}
Suppose that $M_{1}$ and $M_{2}$ are closed, connected and oriented smooth $n$-dimensional manifolds. Then $$H^{i}(M_{1}\# M_{2})\cong H^{i}(M_{1})\oplus H^{i}(M_{2}), \textup{ for } i\neq 0 \textup{ and } n, $$and $H^{0}(M_{1}\# M_{2})\cong H^{n}(M_{1}\# M_{2})\cong \Z$.\end{lemma}

\begin{corollary}
The cohomology of $\#_{i=1}^{n} B_{-p_{i}}$ is as follows:
$$H^{1}(\#_{i=1}^{n} B_{-p_{i}})\cong \Z^{2g_{1}+\cdots +2g_{n}}, \quad H^{2}(\#_{i=1}^{n} B_{-p_{i}})\cong \Z^{2g_{1}+\cdots +2g_{n}}\oplus \Z_{p_{1}}\cdots \oplus \Z_{p_{n}}, $$
and $H^{0}(\#_{i=1}^{n} B_{-p_{i}})\cong H^{3}(\#_{i=1}^{n} B_{-p_{i}})\cong \Z$.

\end{corollary}

\begin{proposition}
\label{cohomology}
For the cobordism $\bar{X}$, its second cohomology has the following  form: $$H^{2}(\bar{X})\cong H^{2}(\#_{i=1}^{n} B_{-p_{i}}) \cong \mathbb{Z}^{2g_{1}+\cdots +2g_{n}}\oplus\mathbb{Z}_{p_{1}}\oplus \cdots \oplus \mathbb{Z}_{p_{n}}.$$ 

\end{proposition}

\begin{proof}
We use the Mayer-Vietoris sequence to compute the cohomology of $\bar{X}$. Observe that $W$ is the union of $\bar{X}$ and boundary connected sum $\mathfrak{D}$, and the intersection of $\bar{X}$ and $\mathfrak{D}$ is $\#_{i=1}^{n} B_{-p_{i}}$. Then we have the following long exact sequence: 
$$0 \rightarrow H^{1}(W)\rightarrow H^{1}(\bar{X})\oplus H^{1}(\Sigma'_{1})\cdots \oplus H^{1}(\Sigma'_{n})\xrightarrow{i^{\ast}} \oplus_{i=1}^{n }H^{1}(B_{-p_{i}})\rightarrow H^{2}(W)$$
$$\rightarrow \oplus_{i=1}^{n} H^{2}(\Sigma'_{i})\oplus H^{2}(\bar{X})\rightarrow H^{2}(\#_{i=1}^{n} B_{-p_{i}}) \rightarrow H^{3}(W)\rightarrow H^{3}(\bar{X})\rightarrow H^{3}(\#_{i=1}^{n} B_{-p_{2}})\rightarrow 0 .$$

Recall that $W$ is homotopy equivalent to $S^{2}\vee \cdots \vee S^{2}$. Then $H^{1}(W)\cong H^{3}(W)\cong H^{4}(W)\cong 0$. Thus,  we have
$$H^{3}(\bar{X})\cong H^{3}(\#_{i=1}^{n} B_{-p_{i}})\cong \Z,\quad  H^{4}(\bar{X})=0,$$
and 
$$0\rightarrow H^{1}(\bar{X})\oplus \mathbb{Z}^{2g_{1}}\cdots \oplus \mathbb{Z}^{2g_{n}}\xrightarrow{i^{\ast}} \oplus_{i=1}^{n}\mathbb{Z}^{2g_{i}} \rightarrow \mathbb{Z}^{n}\rightarrow H^{2}(\bar{X})\oplus \mathbb{Z}^{n}\rightarrow \oplus_{i=1}^{n}(\mathbb{Z}^{2g_{i}}\oplus \mathbb{Z}_{p_{i}})\rightarrow 0.$$

We claim that $H^{1}(\bar{X})=0$. Let $j^{\ast}: H^{1}(\Sigma'_{i})\rightarrow H^{1}(B_{-p_{i}})$. Observe that $H^{1}(\Sigma'_{i})\cong H_{1}(\Sigma'_{i})$ and $H^{1}(B_{-p_{i}})\cong H_{2}(B_{-p_{i}})$ by the Poincare duality. Each generator in $H_{1}(\Sigma'_{i})$ is represented by a simple closed curve in $\Sigma'_{i}$. The curve along with its circle fiber is a generator in $H_{2}(B_{-p_{i}})$. Thus  $H_{1}(\Sigma'_{i})\cong H_{2}(B_{-p_{i}})$. Therefore, $j^{\ast}$ is an isomorphism and $H^{1}(\bar{X})=0$. We have the following short exact sequence:
\begin{equation}
0\rightarrow \mathbb{Z}^{n} \rightarrow H^{2}(\bar{X})\oplus \mathbb{Z}^{n}\rightarrow \oplus_{i=1}^{p}(\mathbb{Z}^{2g_{i}}\oplus \mathbb{Z}_{p_{i}})\rightarrow 0.
\end{equation}

We claim  that $H^{2}(W)\cong \oplus_{i=1}^{n} H^{2}( \Sigma'_{i})$. Each $\mathbb{Z}$-summand in $H_{2}(W, \partial W)\cong H^{2}(W)$ is represented by the surface $\Sigma_{i}$ and it corresponds to the generator of $H^{2}(\Sigma'_{i})\cong H_{0}(\Sigma'_{i})$. Thus $H^{2}(\bar{X})\cong H^{2}(\#_{i=1}^{n} B_{-p_{i}})\cong \oplus_{i=1}^{p}(\mathbb{Z}^{2g_{i}}\oplus \mathbb{Z}_{p_{i}})$.

\end{proof}

\begin{remark}
From the computation in the proof, $\chi(X)=2g_{1}+\cdots 2g_{n}$. 

\end{remark}

\begin{proposition}
The intersection form $\mathcal{Q}$: $H^{2}(\bar{X})/ Tor \times H^{2}(\bar{X})/ Tor \rightarrow \mathbb{Q}$ vanishes.

\end{proposition}

\begin{proof}
For two elements $s, t\in H^{2}(\bar{X})/Tor \cong H_{2}(\bar{X})$, $\mathcal{Q}(s, t)=\langle \bar{s}, \textup{PD}(t) \rangle$ where $\bar{s}$ is the image of $s$ under the map $p_{\ast}: H_{2}(\bar{X})\rightarrow H_{2}(\bar{X}, \partial \bar{X})$ induced by the projection and $\textup{PD}(t)\in H^{2}(\bar{X}, \partial \bar{X})$. By Proposition \ref{cohomology}, the map $i_{\ast}: H_{2}(\partial \bar{X})\rightarrow H_{2}(\bar{X})$ induced by the inclusion is surjective. Then $p_{\ast}=0$ by the long exact sequence of homology. Hence, $\bar{s}=0$ and $\mathcal{Q}(s, t)=0$. Therefore the intersection form $\mathcal{Q}$ vanishes in $\bar{X}$.

\end{proof}

\begin{corollary}
The signature $\sigma(\bar{X})=0$.

\end{corollary}

By Proposition \ref{cohomology}, $H^{2}(\bar{X})\cong H^{2}(\#_{i=1}^{n} B_{-p_{i}})\cong \mathbb{Z}^{2g_{1}\cdots +2g_{n}}\oplus \mathbb{Z}_{p_{1}}\cdots \oplus \mathbb{Z}_{p_{n}}$, and the restriction map to $H^{2}(S^{3}_{-p_{1}, \cdots, -p_{n}}(\L))$ is the projection onto the torsion factors. For $\bm{s}=(s_{1}, \cdots, s_{n})\in \Z^{n}/  \Z^{n}\Lambda$, it corresponds to a  Spin$^{c}$-structure $\mathfrak{s}$ on $S^{3}_{-p_{1}, \cdots, -p_{n}}(\L)$ which can be extended to $W$ by Proposition \ref{spin extension}. We denote its restrictions to $\bar{X}$ and $\#_{i=1}^{n} B_{-p_{i}}$ both by $\mathfrak{s}'$. Moreover, we let $s'_{i}$ denote the  restriction of the Spin$^{c}$-structure on $\bar{X}$ to $B_{-p_{i}}$. By an  argument similar to the one  in \cite[Lemma 3.1]{Ras}, we have  $c_{1}(s'_{i})=2s_{i}$. 

Rasmussen proved that each circle bundle $B_{-p_{i}}$ is standard since 
$$HF^{\infty}(B_{-p_{i}}, s'_{i})\cong HF^{\infty}(\#^{2g} S^{1}\times S^{2}, \mathfrak{s}_{0})$$
where $s'_{i}$ is a torsion Spin$^{c}$-structure on $B_{-p_{i}}$ and $\mathfrak{s}_{0}$ is the unique torsion Spin$^{c}$-structure on $\#^{2g}(S^{1}\times S^{2})$ \cite{Ras}. Thus $HF^{\infty}(\#_{i=1}^{n} B_{-p_{i}}, \mathfrak{s}')$ is also standard, and  by the additivity property of the $d$-invariants \cite[Proposition 4.3]{Lev}, 
$$d(\#_{i=1}^{n} B_{-p_{i}}, \mathfrak{s}')=d(B_{-p_{1}}, s'_{1})+\cdots+d(B_{-p_{n}}, s'_{n}). $$

 Next,  we can use Proposition \ref{standard}  to prove the following $d$-invariant inequality:

\begin{proposition}
\label{cobordism inequality}
$$d(S^{3}_{-p_{1}, \cdots, -p_{n}}(\L), \mathfrak{s})\leq \sum\limits_{i=1}^{n} d(B_{-p_{i}}, s'_{i})+g_{1}+\cdots +g_{n}.$$

\end{proposition}

\begin{proof}
The argument is similar to the one in \cite[Lemma 3.3]{Ras}. Consider the map $F^{+}_{\bar{X}, \mathfrak{s}'}: HF^{+}(S^{3}_{-p_{1}, \cdots, -p_{n}}(\L), \mathfrak{s})\rightarrow HF^{+}(\#_{i=1}^{n} B_{-p_{i}}, \mathfrak{s}')$. This map is $U$-equivariant and agrees with $F^{\infty}_{\bar{X}, s'}$ in high degrees, which implies that it takes $\pi(HF^{\infty}(S^{3}_{-p_{1}, \cdots, -p_{n}}(\L), \mathfrak{s}))$ onto $\pi(\mathfrak{A}_{\mathfrak{s}'})$ by Proposition \ref{standard}. Thus if $1\in \pi(HF^{\infty}(S^{3}_{-p_{1}, \cdots, -p_{n}}(\L), \mathfrak{s})\otimes \mathbb{F}\cong \mathbb{F}[U^{-1}]$ is the element with the lowest absolute grading, we must have 
$$gr(F^{+}_{\bar{X}, \mathfrak{s}'}(1))\leq d(\#_{i=1}^{n} B_{-p_{i}}, \mathfrak{s}').$$

Observe that $gr(1)=d(S^{3}_{-p_{1}, \cdots, -p_{n}}(\L), \mathfrak{s})$ and $F^{+}_{\bar{X}, \mathfrak{s}'}$ shifts the absolute grading by 
$$deg F_{\bar{X}, \mathfrak{s}'}^{+}=\dfrac{ c_{1}^{2}(\mathfrak{s}')-2\chi(\bar{X})-3\sigma(\bar{X})}{4}=\dfrac{0-2(2g_{1}+\cdots+ 2g_{n})-3\cdot 0}{4}=-g_{1}-\cdots -g_{n} $$
where $c_{1}^{2}(\mathfrak{s}')=\mathcal{Q}(c_{1}(\mathfrak{s}'), c_{1}(\mathfrak{s}'))=0$. Thus, 
$$d(S^{3}_{-p_{1}, \cdots, -p_{n}}(\L), \mathfrak{s})\leq \sum\limits_{i=1}^{n} d(B_{-p_{i}}, s'_{i})+g_{1}+\cdots +g_{n}.$$

\end{proof}

Let $\L^{\ast}$ denote the mirror of $\L$. Observe that $S^{3}_{-p_{1}, \cdots, -p_{n}}(\L)$ is obtained from $S^{3}_{p_{1}, \cdots, p_{n}}(\L^{\ast})$ by reversing the orientation. For any $\bm{s}\in \Z^{n}$, choose sufficiently large $p_{i}\gg 0$ so that $\bm{s}\in P_{\H}(\Lambda)$. Let $\mathfrak{s}$ denote the Spin$^{c}$-structure on $S^{3}_{-p_{1}, \cdots, -p_{n}}(\L)$ corresponding to $\bm{s}$. 
By Theorem \ref{d-invariant formula}
$$d(S^{3}_{-p_{1}, \cdots, -p_{n}}(\L), \mathfrak{s})=-d(S^{3}_{p_{1}, \cdots, p_{n}}(\L^{\ast}), \mathfrak{s})=\deg F_{\mathcal{U}', \mathfrak{s}}+2H_{L^{\ast}}(\bm{s}),$$
where $H_{\L^{\ast}}$ is the $H$-function of $\L^{\ast}$.  Let $O$ denote the unlink with $n$ components. Similarly, we have
$$d(S^{3}_{-p_{1}, \cdots, -p_{n}}(O), \mathfrak{s})=-d(S^{3}_{p_{1}, \cdots, p_{n}}(O), \mathfrak{s})=\deg F_{\mathcal{U}', \mathfrak{s}}+2H_{O}(\bm{s}). $$
Thus, $$d(S^{3}_{-p_{1}, \cdots, -p_{n}}(\L), \mathfrak{s})-d(S^{3}_{-p_{1}, \cdots, -p_{n}}(O), \mathfrak{s})=2H_{L^{\ast}}(\bm{s})-2H_{O}(\bm{s})=2h_{\L^{\ast}}(\bm{s}). 
$$

Recall that for a circle bundle $B_{-m}$ with Euler number $-m$ over a closed, oriented genus $g$ surface, $H^{2}(B_{-m})\cong \Z^{2g}\oplus \Z_{m}$. There is a natural way to label the torsion Spin$^{c}$-structures on $B_{-m}$. Let $\mathfrak{t}_{k}$ denote the  torsion Spin$^{c}$-structure on $B_{-m}$ such that $c_{1}(\mathfrak{t}_{k})=2k\in \Z_{m}$ where $-m/2\leq k\leq m/2$.

\begin{proposition} \cite[Proposition 3.4]{Ras}
\label{circle bundle invariant}
Let $B_{-m}$ denote a circle bundle equipped with a torsion Spin$^{c}$-structure $\mathfrak{t}_{k}$ over a closed oriented surface $\Sigma_{g}$.  For $m\gg 0$, 
$$
d(B_{-m}, \mathfrak{t}_{k}) = \left\{
        \begin{array}{ll}
            E(m, k)-g+2\left\lceil \dfrac{g-|k|}{2}\right\rceil & \quad |k|\leq g \\
            E(m, k)-g & \quad |k| > g
        \end{array}
    \right.
$$
where  $\lbrace E(m, k)\mid k\in \mathbb{Z}_{m} \rbrace$ is the set of $d$-invariants of the lens space $L(m, 1)$. 

\end{proposition}

\begin{proof}
The circle bundle $B_{-m}$ can be obtained by doing $-m$-surgery on  ``the Borromean knot" $B\subset \#^{2g}(S^{1}\times S^{2})$. For the large surgery, $HF(B_{-m}, \mathfrak{t}_{k})$ can be obtained from the knot Floer homology $HFK^{\infty}(B)$. Ozsv\'ath and Szab\'o proved that \cite{SZ5}
$$\widehat{HFK}(B, i)\cong \wedge^{g+i}(H^{1}(\Sigma_{g})) \quad \textup { if } |i|\leq g.$$
Otherwise, $\widehat{HFK}(B, i)=0$. The rest of the  proof can be found in \cite[Proposition 3.4]{Ras}.

\end{proof}

\subsection{Proofs of main theorems}
\label{proof of unlink }
We prove Theorem \ref{d-invariant inequality} and Theorem \ref{unlink} in this subsection.

\medskip
\noindent
{\bf Proof of Theorem \ref{d-invariant inequality}: }
By Proposition \ref{cobordism inequality} and Proposition \ref{circle bundle invariant}, 
$$d(\Sp, \mathfrak{s})\leq \sum\limits_{i=1}^{n} (E(p_{i}, s_{i})-g_{i}+2f_{g_{i}}(s_{i}))+g_{1}+\cdots +g_{n}.$$
Recall that $d(\Sp, \mathfrak{s})=2h_{\L^{\ast}}(\bm{s})+d(\Sph, \mathfrak{s})$. For lens spaces, our orientation convention is the one used in \cite{Ras}, namely, that $-p$ surgery on the unknot produces the oriented space $L(p, 1)$. Then $\Sph=L(p_{1}, 1)\# \cdots  \# L(p_{n}, 1)$, and 
$$d(\Sph, \mathfrak{s})=\sum\limits_{i=1}^{n} d(L(p_{i}, 1), s_{i})=\sum\limits_{i=1}^{n} E(p_{i}, s_{i}). $$
Hence, 
$$h_{\L^{\ast}}\leq \sum\limits_{i=1}^{n} f_{g_{i}}(s_{i}).$$
The surfaces $\Sigma_{i}$ bounded by $L_{i}$ are pairwise disjoint. Then the  corresponding link components of the mirror link $\L^{\ast}$  bound the  mirrors of  $\Sigma_{i}$ which have the same genera as $\Sigma_{i}$. Thus we have 
$h_{L}(\bm{s})\leq \sum_{i=1}^{n} f_{g_{i}}(s_{i}).$
\qed

\begin{corollary}
\label{contained as a subset}
For an oriented $n$-component link $\L\subset S^{3}$, $\G(\L)\subset\G_{HF}(\L)$.
\end{corollary}

\begin{proof}
Suppose that the link components of $\L$ bound pairwise disjoint surfaces in $B^{4}$ of genera $g_{i}$. By Theorem \ref{d-invariant inequality}, $h_{\L}(\bm{s})=0$ if $\bm{s}\succeq \bm{g}$ where $\bm{g}=(g_{1}, \cdots, g_{n})$. 

\end{proof}

\begin{definition}
An oriented $n$-component link $\L\subset S^{3}$ is (smoothly) \emph{slice} if there exist $n$ disjoint, smoothly embedded disks in $B^{4}$ with boundary $\L$. 
\end{definition}

\noindent
{\bf Proof of Theorem \ref{unlink}: }
If $\L$ is slice, then $h_{\L}=0$ by Theorem \ref{d-invariant inequality}. Thus $H_{\L}(\bm{v})=H_{O}(\bm{v})=\sum_{i=1}^{n}H(v_{i})$ where $H(v_{i})$ is the $H$-function for the unknot and $\bm{v}=(v_{1}, \cdots, v_{n})\in \Z^{n}$.  We  claim that $HFL^{-}(\L, \bm{v})=0$ if there exists a component $v_{j}>0$. By Proposition \ref{spectral sequence 1}, there exists a spectral sequence converging to $HFL^{-}(\L)$ with the $E_{1}$-page
$$E_{1}(\bm{v})=\bigoplus_{B\subset \lbrace 1, \cdots, n \rbrace} H_{\ast}(A^{-}(\L, \bm{v}-\bm{e}_{B})).$$ 
Let $\mathcal{K}=\lbrace 1, \cdots, n \rbrace \setminus \lbrace j \rbrace$, and 
$$E'(\bm{v})=\bigoplus_{B\subset \mathcal{K}} H_{\ast}(A^{-}(\L, \bm{v}-\bm{e}_{B})),\quad  E''(\bm{v})=\bigoplus_{B\subset \mathcal{K}} H_{\ast}(A^{-}(\L, \bm{v}-\bm{e}_{B}-\bm{e}_{j})).$$
Then $E_{1}(\bm{v})=E'(\bm{v})\oplus E''(\bm{v})$. Recall that for each $B\subset \lbrace 1, \cdots, n \rbrace$, $H_{\ast}(A^{-}(\L, \bm{v}-\bm{e}_{B}))\cong \F[U]$, \cite{Liu}. Let  $\partial_{1}, \partial', \partial ''$ denote  the differentials in $E_{1}(\bm{v}), E'(\bm{v})$ and $E''(\bm{v})$, respectively. Let $z$ denote the generator of $H_{\ast}(A^{-}(\L, \bm{v}-e_{B}-e_{j}))\in E''(\bm{v})$ with homological grading $-2H(\bm{v}-e_{B}-e_{j})$. Observe that $H(\bm{v}-e_{B}-e_{j})=H(\bm{v}-e_{B})$ since $H(v_{j}-1)=H(v_{j})$ for $v_{j}>0$. Then $\partial_{1}(z)=\partial''(z)+z'$ where $z'$ is the generator of $H_{\ast}(A^{-}(\L, \bm{v}-\bm{e}_{B}))$ with homological grading $-2H(\bm{v}-e_{B})$. Let $\mathcal{D}$ be an acyclic chain complex with two generators $a$ and $b$, and the differential $\partial_{D}(a)=b$. Then the chain complex $(E_{1}(\bm{v} ), \partial_{1} )$  is isomorphic to $(E'(\bm{v})\otimes \mathcal{D}, \partial_{1}\otimes \partial_{D})$. Thus $E_{2}=0$, and the spectral sequence collapes at $E_{2}$. Therefore, $HFL^{-}(\L, \bm{v})=0$ if there exists $v_{j}>0$. 

We also have $\widehat{HFL}(\L, \bm{v})=0$ if there exists $v_{j}>0$ by the spectral sequence in  Proposition \ref{spectral sequence 2}. By the symmetric property \cite{SZ5}, $\widehat{HFL}(\L, -\bm{v})=\widehat{HFL}(\L, \bm{v})=0$. Hence, $\widehat{HFL}(\L, \bm{v})=0$ if $\bm{v}\neq \bm{0}$. The dual Thurston polytope of $\L$ is a point at the origin \cite[Theorem 1.1]{SZ6}. Thus, the link $\L$ bounds disjoint disks in $S^{3}$, and $\L$ is an unlink. 
\qed

\subsection{Lower bounds for the 4-genera}
In this subsection, we use $\L\subset S^{3}$ to denote an $n$-component link with vanishing pairwise linking numbers. The inequality in Theorem \ref{d-invariant inequality} produces some lower bounds for the 4-genus of  $\L$. 

\begin{corollary}
For the link $\L$ , 
\begin{equation}
\label{bound 3}
g_{4}(\L)\geq \min\lbrace s_{1}+\cdots+s_{2}\mid h(\bm{x})=0 \textup{ if }\bm{x}\succeq \bm{s}=(s_{1}, \cdots, s_{n}) \rbrace. 
\end{equation}\
\end{corollary}
\begin{proof}
This is straightforward from Corollary \ref{contained as a subset}
\end{proof}

\begin{corollary}
For the link $\L$, 
$g_{4}(\L)\geq 2\max_{\bm{s}\in \Z^{n}} h_{\L}(\bm{s})-n.$
In particular, 
\begin{equation}
\label{bound 1}
g_{4}(\L)\geq 2h_{\L}(\bm{0})-n. 
\end{equation}

\end{corollary}
\begin{proof}
By Theorem \ref{d-invariant inequality}, for all $\bm{s}\in \Z^{n}$,
$h_{\L}(\bm{s})\leq \lceil g_{1}/2 \rceil+\cdots+\lceil g_{n}/2\rceil.$
Observe that $\lceil g_{i}/2 \rceil\leq (g_{i}+1)/2$. Then 
$$g_{1}+\cdots+g_{n}+n\geq 2\max_{\bm{s}\in \Z^{n}} h_{\L}(\bm{s}).$$
Hence $g_{4}(\L)\geq 2\max_{\bm{s}\in \Z^{n}} h_{\L}(\bm{s})-n.$
\end{proof}

\begin{corollary}
Let $g_{4}(L_{i})$ denote the 4-genus of the link component $L_{i}$. Then 
\begin{equation}
\label{bound 2}
g_{4}(\L) \geq 2h_{\L}(\bm{s})-n+|s_{1}|+\cdots+|s_{2}|,
\end{equation}
where $\bm{s}=(s_{1}, \cdots, s_{n})$, and $|s_{i}|\leq g_{4}(L_{i})$. 
\end{corollary}

\begin{proof}
Suppose that $\L$ bound pairwise disjoint surfaces $\Sigma_{i}$ in $B^{4}$ of genera $g_{i}$. Then $g_{i}\geq g_{4}(L_{i})$ for all $i$. If $|s_{i}|\leq g_{4}(L_{i})$, then by Theorem \ref{d-invariant inequality},
$$h_{\L}(\bm{s})\leq\sum_{i=1}^{n}\left \lceil \dfrac{g_{i}-|s_{i}|}{2} \right\rceil.$$
Since $\lceil (g_{i}-|s_{i}|)/2\rceil\leq (g_{i}-|s_{i}|+1)/2$, we have $g_{1}+\cdots+g_{n}\geq 2h_{\L}(\bm{s})-n+|s_{1}|+\cdots+|s_{2}|$. Hence, $g_{4}(\L)\geq 2h_{\L}(\bm{s})-n+|s_{1}|+\cdots+|s_{2}|$.
\end{proof}

For the rest of the subsection, we prove that the analogues of Lemma \ref{shape 1}, Lemma \ref{shape 2} and Lemma \ref{shape 3} hold for
the set $\G(\L)$. For an oriented link $\L$ with vanishing pairwise linking numbers, we use the \emph{cancellation process} to find pairwise disjoint surfaces in $B^{4}$ bounded by $\L$. Let $\Sigma_{i}\subset S^{3}$ denote a Seifert surface bounded by $L_{i}$. Then $\Sigma_{i}$ and $\Sigma_{j}$ intersect transversely  at even number of  points in $B^{4}$ since the linking number equals 0. We remove the tubular neighborhoods of  a positive crossing and a negative crossing in $\Sigma_{i}$ and obtain a new surface with two punctures. Add a tube along  an arc in $\Sigma_{j}$  which connects the two intersection points to the punctured surface where the attaching circles are boundaries of these two punctures, as in  Figure \ref{cancel}. Then we obtain a new surface $\Sigma_{i}^{'}$  with fewer intersection points with $\Sigma_{j}$ and higher genus compared with $\Sigma_{i}$. The tube can also be attached to the surface $\Sigma_{j}$ along an arc connecting the intersection  points in $\Sigma_{i}$. We  repeat the process  until we get pairwise disjoint surfaces in $B^{4}$ bounded by $\L$.  We call the process of adding tubes to eliminate intersection points as \emph{cancellation process}.
\begin{figure}[H]
\centering
\includegraphics[width=3.0in]{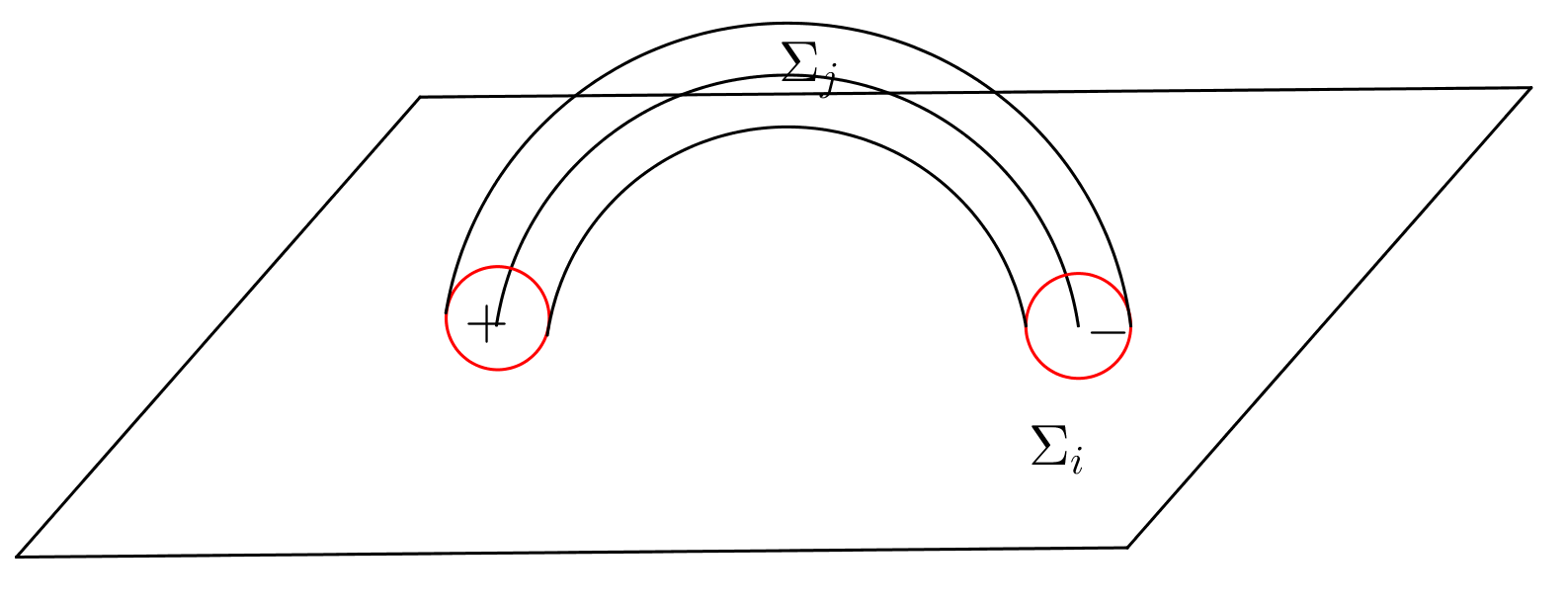}
\caption{Cancellation process \label{cancel}}
\end{figure}

\begin{lemma}
\label{shape 11}
If $\bm{g}\in \G(\L)$ and $\bm{y}\succeq \bm{g}$, then $\bm{y}\in \G(\L)$. Equivalently, if $\bm{g}\notin \G(\L)$ and $\bm{y}\preceq \bm{g}$, then $\bm{y}\notin \G(\L)$.

\end{lemma}

\begin{proof}
If $\bm{g}=(g_{1}, \cdots, g_{n})\in \G(\L)$, there exist pairwise disjoint surfaces $\Sigma_{i}$ embedded in $B^{4}$ of genera $g_{i}$ and $\partial \Sigma_{i}=L_{i}$. We can attach tubes to the surfaces $\Sigma_{i}$ to increase the genera. Thus $\bm{y}\in \G(\L)$ if $\bm{y}\succeq \bm{g}$. 

\end{proof}

\begin{lemma}
\label{shape 22}
If $\bm{g}=(g_{1}, \cdots, g_{n})\in \G(\L)$, then $\bm{g}\setminus g_{i}\in \G(\L\setminus L_{i})$ for all $1\leq i\leq n$. Moreover, if $\bm{g}\setminus g_{i}\in \G(\L\setminus L_{i})$, then for $g_{i}$ sufficiently large, $\bm{g}=(g_{1}, \cdots, g_{n})\in \G(\L)$.

\end{lemma}

\begin{proof}
 If $\bm{g}\in \G(\L)$, it is easy to obtain that $\bm{g}\setminus g_{i}\in \G(\L\setminus L_{i})$. Conversely, if $\bm{g}\setminus g_{i}\in \G(\L \setminus L_{i})$, for sufficiently large $g_{i}\gg 0$, we claim that $\bm{g}\in \G(\L)$. Suppose that $\L\setminus L_{i}$ bounds pairwise disjoint surfaces $\Sigma_{j}$ in $B^{4}$. Let $\Sigma_{i}$ in $S^{3}$ denote a Seifert surface bounded by $L_{i}$. Then $\Sigma_{i}$ intersects with $\Sigma_{j}$ transversely at even number of points in $B^{4}$ since the linking number equals 0. By the cancellation process, we add tubes to $\Sigma_{i}$ until the new surface is disjoint from all the surfaces $\Sigma_{j}$. Thus for sufficiently large $g_{i}$, $\bm{g}\in \G(\L)$. 

\end{proof}

\begin{lemma}
The set $\G(\L)$ is determined by the set of maximal lattice points and $\G(\L\setminus L_{i})$ for all $1\leq i\leq n$. 
\end{lemma}
\begin{proof}
The proof is similar to the one in Lemma \ref{shape 3} by using Lemma \ref{shape 11} and Lemma \ref{shape 22}. 
\end{proof}

\section{Examples}
\label{examples}

\subsection{Examples}
\label{simple examples}
For $L$-space links, the $H$-function can be computed explicitly by the Alexander polynomials of the link and sublinks. The lower bounds for 4-genus of the link in Section \ref{main proof} can also be computed explicitly. In this section, we will show examples of $L$-space links where $\G(\L)=\G_{HF}(\L)$.


\begin{example}
\label{family of two bridge}
Let $k$ be a positive integer. The  two bridge link $\L_{k}=b(4k^{2}+4k, -2k-1)$  is a 2-component $L$-space link with linking number $0$ \cite{Liu}, and both link components are unknots,  see Figure \ref{two bridge} . The Alexander polynomial of $\L_{k}$ is computed in \cite[Section 3]{Liu}, 
$$\Delta_{\L_{k}}(t_{1}, t_{2})=(-1)^{k}\sum\limits_{|i+1/2|+|j+1/2|\leq k} (-1)^{i+j}t_{1}^{i+1/2}t_{2}^{j+1/2}. $$

\definecolor{linkcolor0}{rgb}{0.45, 0.15, 0.15}
\definecolor{linkcolor1}{rgb}{0.15, 0.15, 0.45}
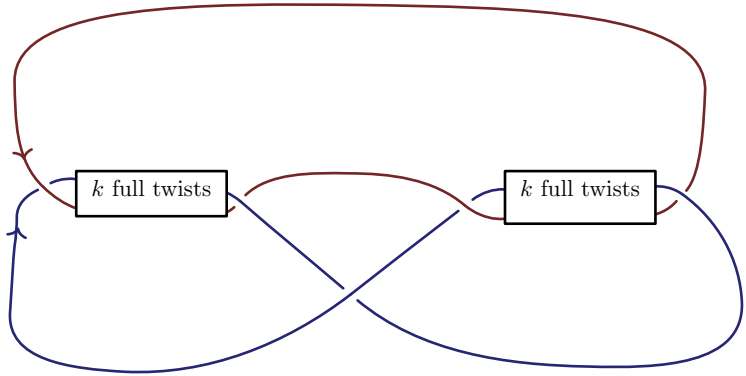
\begin{figure}[H]
\begin{tikzpicture}[line width=1, line cap=round, line join=round]
  \begin{scope}[color=linkcolor0]
    \draw (9.11, 2.54) .. controls (9.31, 2.77) and (9.34, 3.40) .. 
          (9.36, 3.86) .. controls (9.41, 4.95) and (6.83, 4.98) .. 
          (4.79, 5.01) .. controls (2.76, 5.03) and (0.18, 5.06) .. 
          (0.18, 4.00) .. controls (0.19, 3.49) and (0.19, 2.94) .. (0.56, 2.59);
    \draw (0.56, 2.59) .. controls (0.83, 2.33) and (1.20, 2.15) .. (1.42, 2.39);
    \draw (1.57, 2.55) .. controls (1.77, 2.78) and (2.13, 2.64) .. (2.38, 2.40);
    \draw (2.38, 2.40) .. controls (2.60, 2.21) and (2.91, 2.14) .. (3.10, 2.32);
    \draw (3.25, 2.48) .. controls (3.56, 2.78) and (4.13, 2.78) .. 
          (4.62, 2.77) .. controls (5.17, 2.77) and (5.75, 2.70) .. (6.17, 2.34);
    \draw (6.17, 2.34) .. controls (6.42, 2.14) and (6.76, 2.09) .. (6.99, 2.29);
    \draw (7.16, 2.46) .. controls (7.42, 2.71) and (7.83, 2.60) .. (8.12, 2.36);
    \draw (8.12, 2.36) .. controls (8.38, 2.14) and (8.76, 2.15) .. (8.98, 2.40);
    \draw[->] (0.3, 3.00) -- +(0.01, -0.05);
  \end{scope}
  \begin{scope}[color=linkcolor1]
    \draw (4.65, 1.16) .. controls (3.78, 0.49) and (2.74, 0.08) .. 
          (1.65, 0.13) .. controls (0.91, 0.17) and (0.08, 0.28) .. 
          (0.12, 0.92) .. controls (0.14, 1.21) and (0.16, 1.62) .. 
          (0.18, 1.80) .. controls (0.20, 1.93) and (0.21, 2.06) .. 
          (0.21, 2.19) .. controls (0.21, 2.36) and (0.34, 2.48) .. (0.49, 2.55);
    \draw (0.66, 2.64) .. controls (0.93, 2.77) and (1.26, 2.67) .. (1.50, 2.47);
    \draw (1.50, 2.47) .. controls (1.75, 2.26) and (2.09, 2.13) .. (2.31, 2.33);
    \draw (2.46, 2.47) .. controls (2.65, 2.66) and (2.96, 2.58) .. (3.18, 2.40);
    \draw (3.18, 2.40) .. controls (3.63, 2.01) and (4.09, 1.63) .. (4.55, 1.24);
    \draw (4.74, 1.08) .. controls (5.58, 0.38) and (6.71, 0.22) .. 
          (7.80, 0.20) .. controls (8.77, 0.17) and (9.87, 0.22) .. 
          (9.85, 1.04) .. controls (9.83, 1.62) and (9.53, 2.15) .. (9.06, 2.49);
    \draw (9.06, 2.49) .. controls (8.81, 2.67) and (8.47, 2.61) .. (8.22, 2.43);
    \draw (8.02, 2.28) .. controls (7.73, 2.07) and (7.32, 2.11) .. (7.07, 2.38);
    \draw (7.07, 2.38) .. controls (6.87, 2.61) and (6.52, 2.61) .. (6.27, 2.42);
    \draw (6.08, 2.27) .. controls (5.60, 1.90) and (5.12, 1.53) .. (4.65, 1.16);
    \draw[->] (0.22, 2.00) -- +(0.01, 0.05);
  \end{scope}
    \draw[fill=white] (1,2.2) -- (1,2.8) -- (2,2.8) node[scale=0.8,below=0.1] {$k$ full twists} -- (3,2.8) -- (3,2.2) -- cycle;
    \draw[fill=white] (6.7,2.1) -- (6.7,2.8) -- (7.7,2.8) node[scale=0.8,below=0.1] {$k$ full twists} -- (8.7,2.8) -- (8.7,2.1) -- cycle;
\end{tikzpicture}
\caption{Two bridge link $b(4k^{2}+4k, -2k-1)$.    \label{two bridge}}
\end{figure}

The $H$-function of $\L_{k}$ is computed in \cite[Proposition 6.12]{Liu}. Then the $h$-function of $\L_{k}$  is shown  in Figure \ref{h-function of two bridge},  where $h(k, 0)=0$, and $h(k-1, 0)=1$. For the shaded area bounded by the ``stairs", the $h$-function is nonzero, and  $h(\bm{s})=0$ for all lattice points $\bm{s}$ on the ``stairs" and outside of the shaded area in Figure \ref{h-function of two bridge}. Thus, $\G_{HF}(\L)$ consists of all the lattice points on the ``stairs" and outside the shaded area in the first quadrant. By the inequality  \eqref{bound 1},
$$g_{4}(\L_{k})\geq \min\lbrace s_{1}+s_{2}\mid h(\bm{x})=0 \textup{ if }\bm{x}\succeq \bm{s}=(s_{1}, s_{2}) \rbrace=k. $$

Observe that the  components of $\L_{k}$ bound disks $D_{1}$ and $D_{2}$ in $S^{3}$. Push the disks into $B^{4}$. Then they intersect transversely at $2k$ points in $B^{4}$. By the cancellation process of  crossings, we obtain disjoint surfaces $\Sigma'_{1}$ and $\Sigma'_{2}$ in $B^{4}$ bounded by the link components. Assume that  the genus of $\Sigma'_{1}$ is $k$ and $\Sigma'_{2}$ is still a disk of genus 0. Then $g_{4}(\L_{k})\leq k$. Thus, $g_{4}(\L_{k})=k$. We can add tubes to either $D_{1}$ or $D_{2}$ in the cancellation process. Thus, for all $\bm{g}=(g_{1} ,g_{2})$ with $g_{1}+g_{2}=k$, we  find disjoint surfaces in $B^{4}$ of genera $g_{1}, g_{2}$, respectively. Therefore, $\G(\L_{k})=\G_{HF}(\L_{k})$.
\end{example}

\begin{figure}[H]
\centering
\includegraphics[width=4.0in]{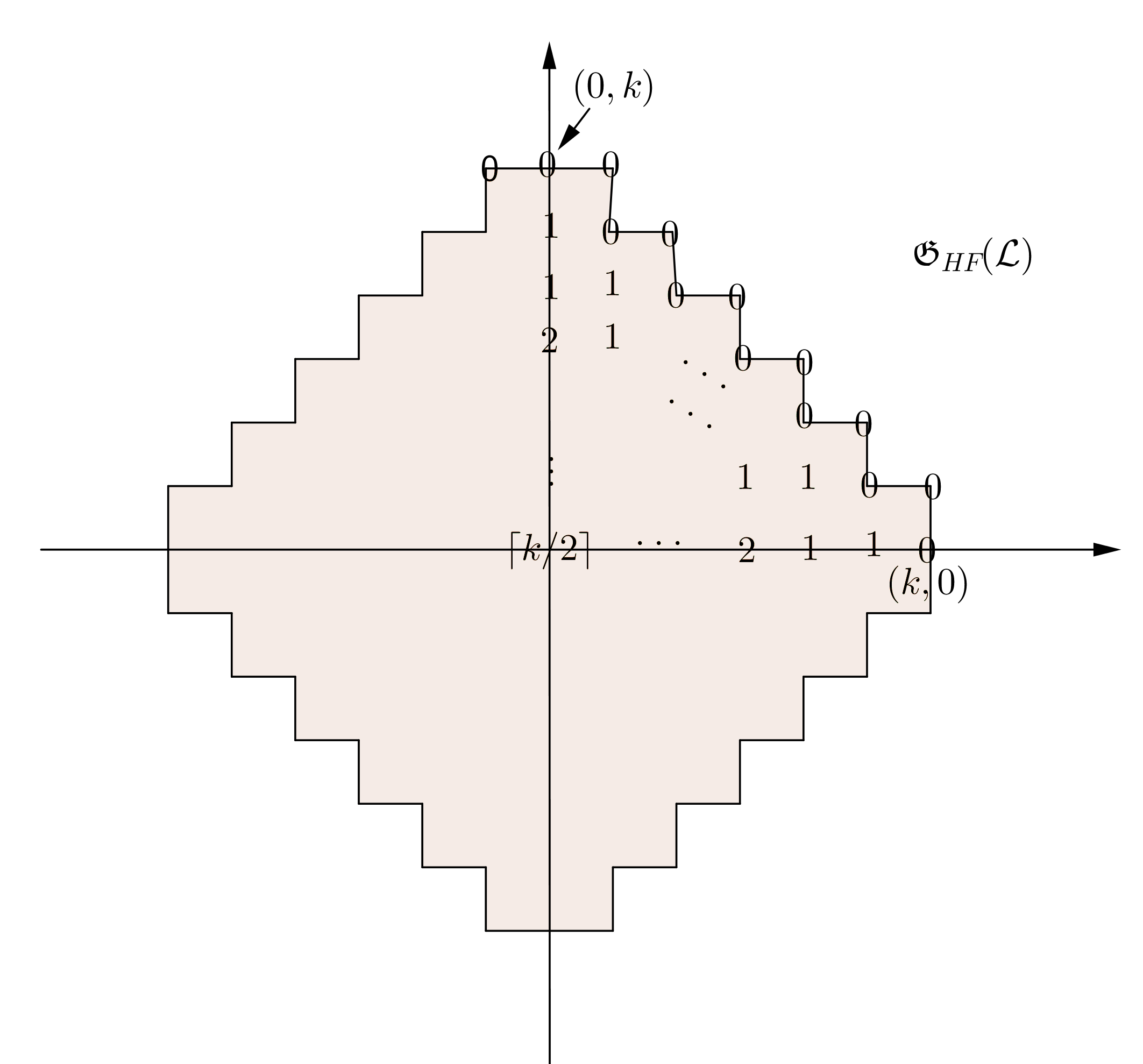}
\caption{ The $h$-function of $\L_{k}$. \label{h-function of two bridge}}
\end{figure}

\begin{remark}
For $k=1$ we get the Whitehead link  $\L_{1}$, and the 4-genus  $g_{4}(\L_{1})=1$.

\end{remark}

\begin{example}
\label{family of simple links}
The Borromean link  $\L=L_{1}\cup L_{2}\cup L_{3}$ is a 3-component $L$-space link with vanishing pairwise linking numbers \cite{Liu}. Its Alexander polynomial equals
$$\Delta_{\L}(t_{1}, t_{2}, t_{3})=(t_{1}^{1/2}-t_{1}^{-1/2})(t_{2}^{1/2}-t_{2}^{-1/2})(t_{3}^{1/2}-t_{3}^{-1/2}).$$
By Equation \eqref{converse computation of h-function}, if $\bm{v}\succ \bm{0}$, $h_{\L}(\bm{v})=0$ and $h_{\L}(\bm{0})=1$. Thus,  $\G_{HF}(\L)=\lbrace \bm{v}\in \Z^{n}\mid \bm{v}\succ \bm{0}\rbrace$.

Each component $L_{i}$ is an unknot and bounds a disk $D_{i}$ in $B^{4}$. The disks $D_{1}$ and $D_{2}$ intersect transversely at two points in $B^{4}\setminus L_{3}$. By the cancellation process, $L_{1}$ and $L_{2}$ bound disjoint surfaces $D'_{1}$ and $D_{2}$ of genera 1 and 0 respectively in $B^{4}\setminus L_{3}$. The disk $D_{3}$ intersects $D'_{1}$ and $D_{2}$ transversely in $B^{4}$. Since $D'_{1}$ and $D_{2}$ are disjoint, by an isotopy in $B^{4}$, we can make $D_{3}$ disjoint from $D'_{1}$ and $D_{2}$. Thus, $g_{4}(\L)=1$, and $\G(\L)=\G_{HF}(\L)$.

\end{example}

\begin{example}
\label{mirror link}
The mirror of $L7a3$ is a 2-component $L$-space link $\L=L_{1}\cup L_{2}$ with linking number 0, where $L_{1}$ is the right-handed trefoil and $L_{2}$ is the unknot \cite{Liu}. Its Alexander polynomial equals 
$$\Delta_{\L}(t_{1}, t_{2})=-(t_{1}^{1/2}-t_{1}^{-1/2})(t_{2}^{1/2}-t_{2}^{-1/2})(t_{2}+t_{2}^{-1}).$$ 
The $h$-function in the first quadrant is shown as in Figure \ref{mirror h-function} by Equation \eqref{converse computation of h-function} or the formula in \cite{Liu}. Then the shaded area is $\mathfrak{G}_{HF}(\L)$ and $g_{4}(\L)\geq 2$. 

Observe that the right-handed trefoil and the unknot bound Seifert surfaces of genera 1 and 0 respectively in $S^{3}$. They intersect transversely at two points in $B^{4}$. By the cancellation process, we can obtain disjoint surfaces of genera $(2, 0)$ or $(1, 1)$ bounded by the link.  Thus, $g_{4}(\L)=2$, and $\G(\L)=\G_{HF}(\L)$.

\end{example}

\begin{figure}[H]
\centering
\includegraphics[width=2.0in]{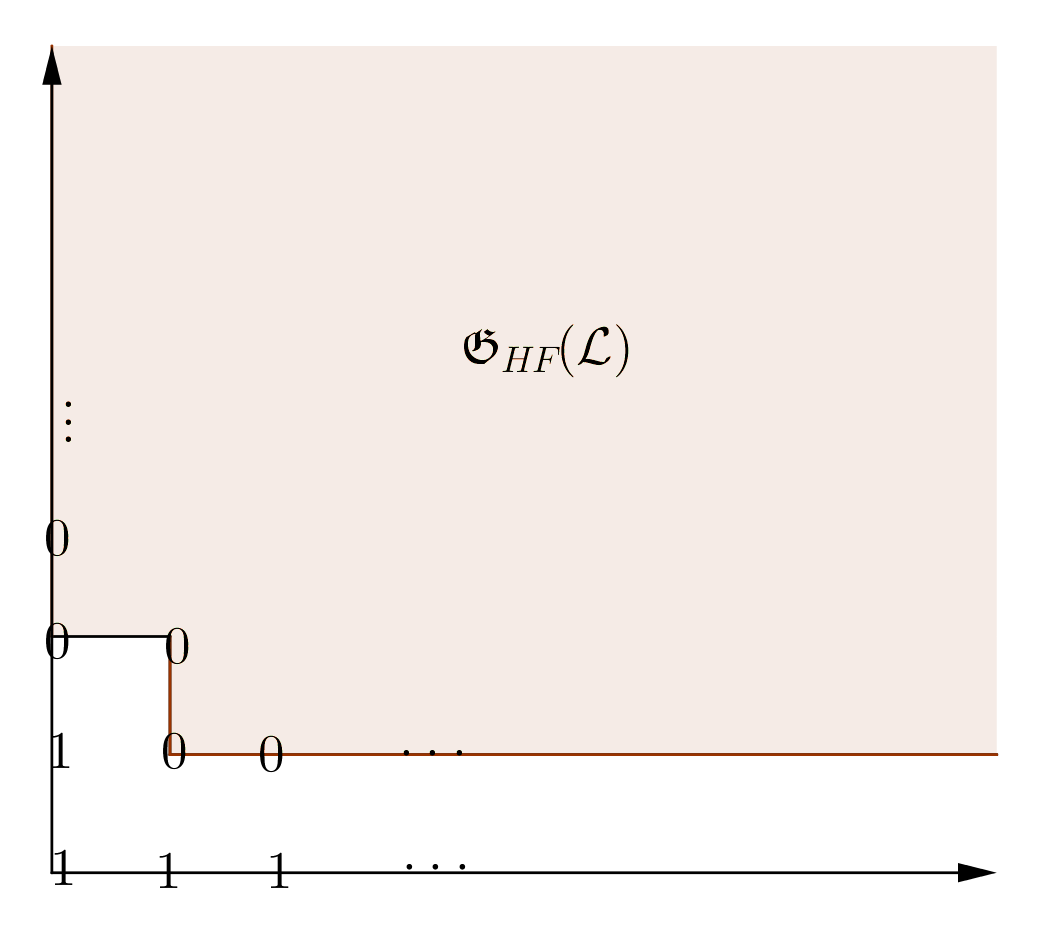}
\caption{The $h$-function for the mirror of L7a3   \label{mirror h-function}}

\end{figure}

\begin{example} Let $\L$ denote the disjoint union of two links $\L_{1}$ and $\L_{2}$. Then $\G(\L)=\G(\L_{1})\times \G(\L_{2})$ and $\G_{HF}(\L)=\G_{HF}(\L_{1})\times \G_{HF}(\L_{2})$.

\end{example}

\begin{proof}
Suppose that $\L_{1}$ has $n_{1}$ components and $\L_{2}$ has $n_{2}$ components. If $\bm{g}_{1}\in \G(\L_{1})$ and $\bm{g}_{2}\in \G(\L_{2})$, then $(\bm{g}_{1}, \bm{g}_{2})\in \G(\L)$ where $\L=\L_{1}\sqcup \L_{2}$. Conversely, if  $$\bm{g}=(g_{1}, \cdots, g_{n_{1}}, \cdots, g_{n_{1}+n_{2}})\in \G(\L),$$
it is straightforward to obtain that $(g_{1}, \cdots, g_{n_{1}})\in \G(\L_{1})$ and $(g_{n_{1}+1}, \cdots, g_{n_{1}+n_{2}})\in \G(\L_{2})$. Thus, $\G(\L)=\G(\L_{1})\times \G(\L_{2})$. 

For the set $\G_{HF}(\L)$, recall that $H_{\L}(\bm{s})=H_{\L_{1}}(\bm{s}_{1})+H_{\L_{2}}(\bm{s}_{2})$ \cite[Proposition 3.11]{Boro} where $\bm{s}=(s_{1}, \cdots, s_{n_{1}}, \cdots, s_{n_{1}+s_{n_{2}}})$,  $\bm{s_{1}}=(s_{1}, \cdots, s_{n_{1}})$  and $\bm{s_{2}}=(s_{n_{1}+1}, \cdots, s_{n_{1}+n_{2}})$. Since  $H$-functions take nonnegative values, $H_{\L}(\bm{s})=0$ if and only if $H_{\L_{1}}(\bm{s}_{1})=H_{\L_{2}}(\bm{s}_{2})=0$. Thus, $\G_{HF}(\L)=\G_{HF}(\L_{1})\times \G_{HF}(\L_{2})$. 
\end{proof}

\subsection{Cables of $L$-space links}
Let $\L=L_{1}\cup \cdots \cup L_{n}\subset S^{3}$ be an $L$-space link with vanishing pairwise linking numbers. For each $1\leq i \leq n$, let $p_{i}$ and $q_{i}$ be coprime  positive integers. Consider the link $\L_{cab}=L_{(p_{1}, q_{1})}\cup \cdots \cup L_{(p_{n}, q_{n})}$  where $L_{(p_{i}, q_{i})}$ is the $(p_{i}, q_{i})$-cable on $L_{i}$. If for all $i$, $p_{i}/q_{i}$ are sufficiently large, then $\L_{cab}$ is also an $L$-space link \cite[Proposition 2.8]{Boro}. 
\begin{theorem}
\label{cable property}
For such cable links $\L_{cab}$, $$\G_{HF}(\L_{cab})=\lbrace \bm{u}\in \Z^{n}\mid \bm{u}\succeq T(\bm{s}) \textup{ for some } \bm{s}\in \G_{HF}(\L)\rbrace$$
where $T: \Z^{n}\rightarrow \Z^{n}$ is defined as $$T(\bm{s})= \bm{p}\cdot \bm{s}+\left( \left( p_{1}-1\right) \left( q_{1}-1\right)/2, \cdots, \left( p_{n}-1\right) \left( q_{n}-1\right)/2\right), $$  $\bm{p}=(p_{1}, \cdots , p_{n}), \bm{s}=\lbrace s_{1}, \cdots, s_{n} \rbrace$ and $\bm{p}\cdot \bm{s}=p_{1}s_{1}+\cdots+p_{n}s_{n}$. 
 \end{theorem}
 
 \begin{proof}
 
By Lemma \ref{determine Alexander polynomial}, $\G_{HF}(\L)$ is determined by Alexander polynomials of the link and sublinks. Let $\Delta_{\L}(t_{1}, \cdots, t_{n})$ denote the Alexander polynomial of $\L$. Then the Alexander polynomial of the cable link $\L_{cab}$ is computed by Turaev in \cite[Theorem 1.3.1]{Tu},
\begin{equation}
\label{Tu}
\Delta_{\L_{cab}}(t_{1}, \cdots, t_{n})=\Delta_{\L}(t_{1}^{p_{1}}, \cdots, t_{n}^{p_{n}})\prod\limits_{i=1}^{n}\dfrac{t_{i}^{p_{i}q_{i}/2}-t_{i}^{-p_{i}q_{i}/2}}{t_{i}^{q_{i}/2}-t_{i}^{-q_{i}/2}}.
\end{equation}
If $\L$ is a knot, we should replace the Alexander polynomials by $\Delta_{\L}(t)/(t-1)$. Then 
\begin{equation}
\label{tu2}
\tilde{\Delta}_{\L_{cab}}(t_{1}, \cdots, t_{n})=t_{1}^{1/2-p_{1}/2}\cdots t_{n}^{1/2-p_{n}/2}\tilde{\Delta}_{\L}(t_{1}^{p_{1}}, \cdots, t_{n}^{p_{n}})\prod\limits_{i=1}^{n}\dfrac{t_{i}^{p_{i}q_{i}/2}-t_{i}^{-p_{i}q_{i}/2}}{t_{i}^{q_{i}/2}-t_{i}^{-q_{i}/2}}.
\end{equation}

Observe that $T(\bm{s})=\bm{p}\cdot \bm{s}+(1/2-p_{1}/2, \cdots, 1/2-p_{n}/2)+((p_{1}q_{1}-q_{1})/2, \cdots, (p_{n}q_{n}-q_{n})/2)$ for any $\bm{s}\in \Z^{n}$. We claim that the coefficients of $t_{1}^{y_{1}}\cdots t_{n}^{y_{n}}$ in $\tilde{\Delta}_{\L}(t_{1}, \cdots, t_{n})$ are 0 for all $\bm{y}\succ \bm{s}$ if and only if for all  $\bm{y'}\succ T(\bm{s})$, the coefficients of $t_{1}^{y'_{1}}\cdots t_{n}^{y'_{n}}$ are 0 in $\tilde{\Delta}_{\L_{cab}}(t_{1}, \cdots, t_{n})$. By Equation \eqref{tu2}, the coefficient of $t_{1}^{y'_{1}}\cdots t_{n}^{y'_{n}}$ is 0 if $\bm{y}'$ is not in the image of $T$ and $\bm{y}'\succ T(\bm{s})$. If $\bm{y}'=T(\bm{y})$ for some $\bm{y}\succ \bm{s}$, then the coefficient of $t_{1}^{y'_{1}}\cdots t_{n}^{y'_{n}}$ equals the coefficient of $t_{1}^{y_{1}}\cdots t_{n}^{y_{n}}$ in $\tilde{\Delta}_{\L}(t_{1}, \cdots, t_{n})$ which is 0. This proves the ``only if " part. For the ``if " part, suppose there exists $\bm{y}\succ \bm{s}$ such that the coefficient of  $t_{1}^{y_{1}}\cdots t_{n}^{y_{n}}$ is nonzero. Then the coefficient corresponding to $T(\bm{y})$ in $\tilde{\Delta}_{\L_{cab}}(t_{1}, \cdots, t_{n})$ is 0 which contradicts to our assumption. For all subsets $B\subset \lbrace 1, \cdots, n\rbrace$, the similar statement holds for Alexander polynomials $\tilde{\Delta}_{\L\setminus L_{B}}(t_{i_{1}}, \cdots, t_{i_{k}})$ and $\tilde{\Delta}_{\L_{cab}\setminus (\L_{cab})_{B}}(t_{i_{1}}, \cdots, t_{i_{k}})$.  By Lemma \ref{determine Alexander polynomial},
$\bm{y}'\in \G_{HF}(\L_{cab})$ if $\bm{y}'\succ T(\bm{s})$ for some $\bm{s}\in \G_{HF}(\L)$. Thus, $\G_{HF}(\L_{cab})\supset \lbrace \bm{u}\in \Z^{n}\mid \bm{u}\succeq T(\bm{s}) \textup{ for some } \bm{s}\in \G_{HF}(\L)\rbrace$.

Conversely, suppose $\bm{y}'\in \G_{HF}(\L_{cab})$. If $\bm{y}'=T(\bm{s})$ for some $\bm{s}\in \Z^{n}$, by Lemma \ref{determine Alexander polynomial} and the claim, $\bm{s}\in \G_{HF}(\L)$. If $\bm{y}'$ is not in the image of $T$, then there exists $\bm{s}\in \Z^{n}$ such that $\bm{y}'\succ T(\bm{s})$ and $\bm{y}'\prec T(\bm{y})$ for all $\bm{y}\succ \bm{s}$. We claim that $\bm{s}\in \G_{HF}(\L)$. If there exists $\bm{y}\succ \bm{s}$ such that the coefficient corresponding to $\bm{y}$ in $\tilde{\Delta}_{\L}(t_{1}, \cdots, t_{n})$ is not 0, then the coefficient corresponding to $T(\bm{y})$ in $\tilde{\Delta}_{\L_{cab}}(t_{1}, \cdots, t_{n})$ is also not 0 which contradicts to our assumption. Similarly, we prove that for all subsets $B\subset \lbrace 1, \cdots, n \rbrace$ and all $\bm{y}\succ \bm{s}$, the coefficients corresponding to $\bm{y}\setminus y_{B}$ in $\tilde{\Delta}_{\L\setminus L_{B}}(t_{i_{1}}, \cdots, t_{i_{k}})$ are all 0. By Lemma \ref{determine Alexander polynomial}, $\bm{s}\in \G_{HF}(\L)$. Thus, $\G_{HF}(\L_{cab})=\lbrace \bm{u}\in \Z^{n}\mid \bm{u}\succeq T(\bm{s}) \textup{ for some } \bm{s}\in \G_{HF}(\L)\rbrace$.

\end{proof}

\begin{lemma}
\label{geometry property}
For such cable links $\L_{cab}$, $\G(\L_{cab})\supset \left\lbrace \bm{u}\in \Z^{n} \mid \bm{u}\succeq T(\bm{g})\textup{ for some } \bm{g}\in \G(\L) \right\rbrace.$
\end{lemma}

\begin{proof}
Suppose that the link components in $\L$ bound pairwise disjoint surfaces $\Sigma_{i}$ in $B^{4}$ of genera $g_{i}$. The cable knot $L_{(p_{i}, q_{i})}$ bounds a surface of genus $p_{i}g_{i}+(p_{i}-1)(q_{i}-1)/2$. We  start with $p_{i}$ copies of $\Sigma_{i}$ and use $(p_{i}-1)(q_{i}-1)$ half-twisted bands to connect them. Since $\Sigma_{i}$ are pairwise disjoint, the new surfaces are also pairwise disjoint. 

\end{proof}

\noindent
{\bf Proof of Proposition \ref{cable sharp}: }
Let $\G'= \left\lbrace \bm{u}\in \Z^{n} \mid \bm{u}\succeq T(\bm{g})\textup{ for some } \bm{g}\in \G(\L) \right\rbrace$. By assumption, $\G(\L)=\G_{HF}(\L)$. Then $\G'=\G_{HF}(\L_{cab})$ by Theorem \ref{cable property}. Since $\G_{HF}(\L_{cab})\supset \G(\L_{cab})\supset \G'$ by Lemma \ref{geometry property}, we have $\G_{HF}(\L_{cab})=\G(\L_{cab})$.
\qed

\begin{remark}
Proposition \ref{cable sharp} also holds if we only replace some link components in $\L$ by their cables. 
 
\end{remark}

By Proposition \ref{cable sharp}, we can apply cables on all $L$-space links in Examples of Subsection \ref{simple examples}.

\begin{example} 
Cables on the Whitehead link. 

Let $Wh_{p, q}$ denote the link consisting of the $(p, q)$-cable on one  component of the Whitehead link  and the unchanged second component. The linking number is 0, and $Wh_{p, q}$ is an $L$-space link if  $p, q$ are coprime with $q/p\geq 3$, \cite{Boro}. 
 
By Proposition \ref{cable sharp}, the $h$-function of $Wh_{p, q}$ in the first quadrant is shown as in Figure \ref{whitehead cable}. The shaded area is $\G_{HF}(Wh_{p, q})=\G(Wh_{p, q})$.  Thus
$g_{4}(Wh_{p, q})=g_{1}+g_{2}= (p-1)(q-1)/2+1.$ The link $Wh_{p, q}$ bounds disjoint surfaces of genera $(p-1)(q-1)/2$ and $1$ as in Figure \ref{whitehead cable}. The red line denotes the tube attached to the disk bounded by the unknot. 

\begin{figure}[H]
\centering
\includegraphics[width=3.0in]{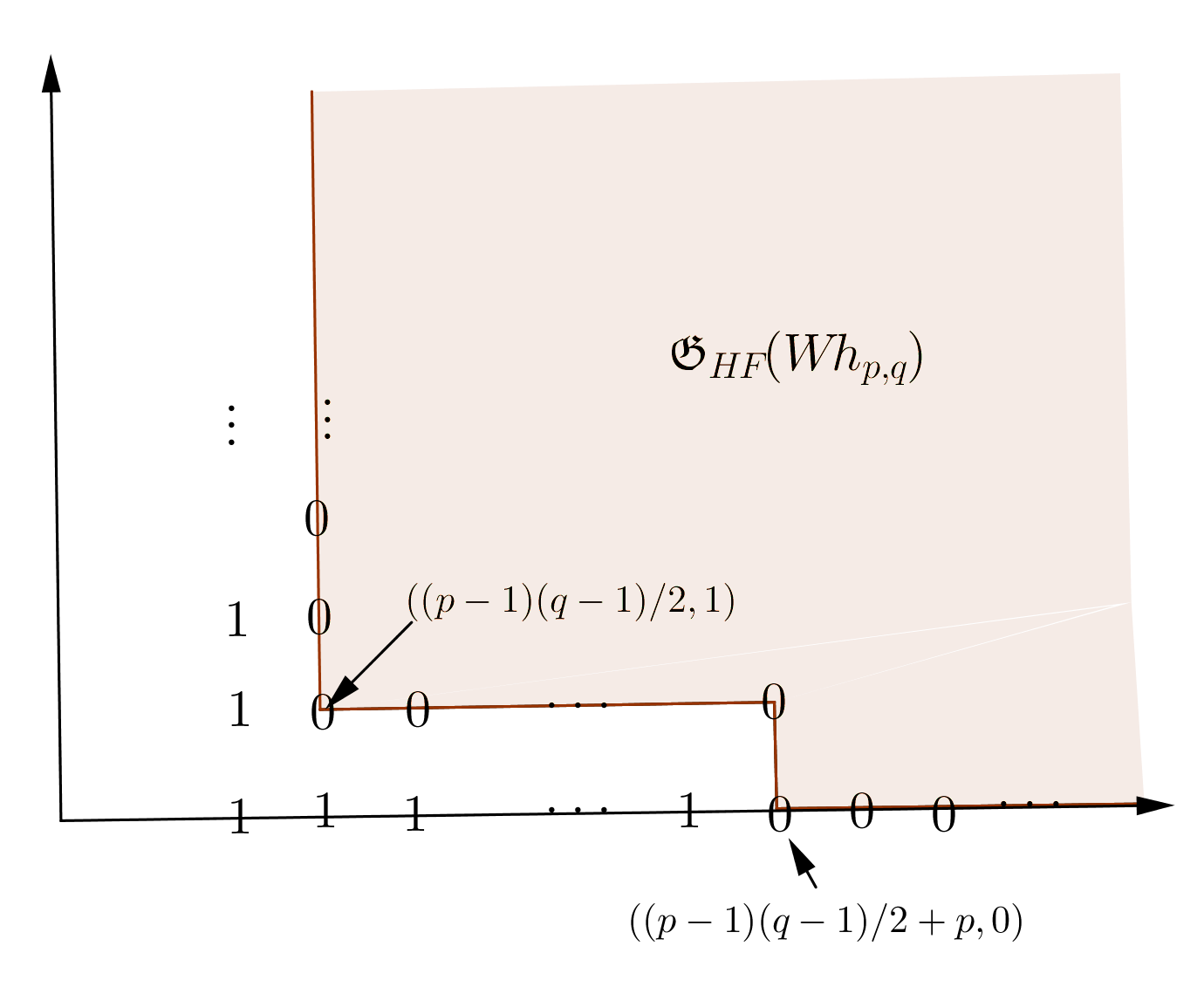}
\includegraphics[width=2.5in]{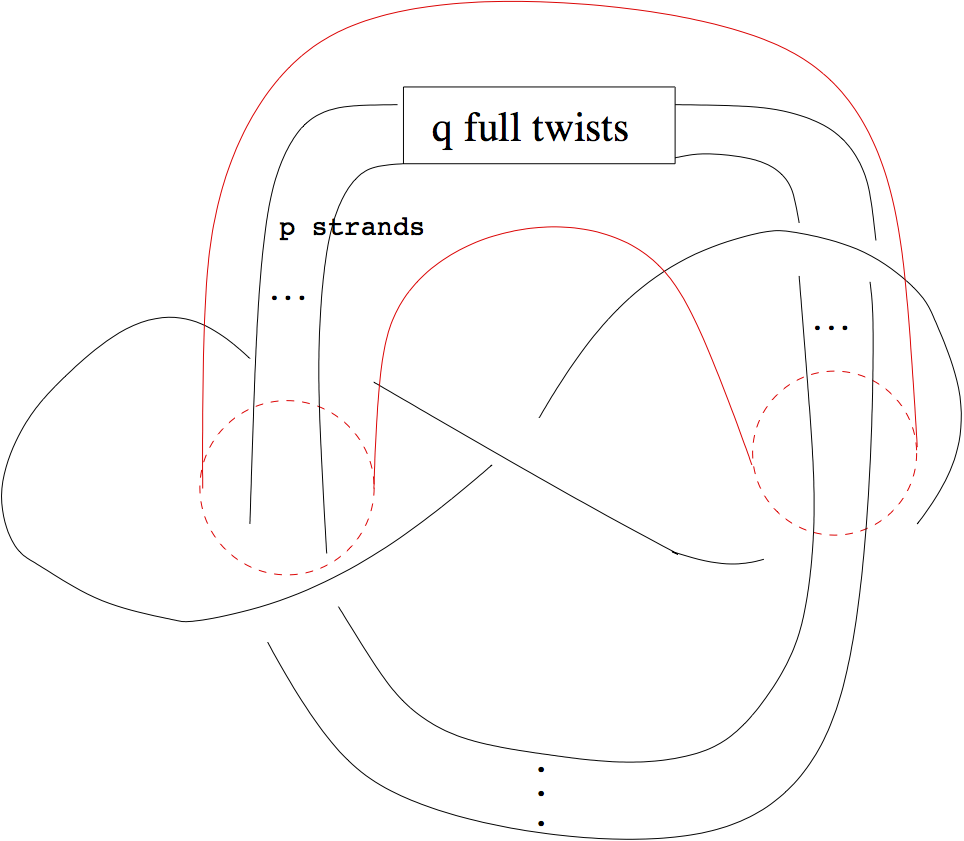}
\caption{The h-function of $Wh_{p, q}$  \label{whitehead cable} }

\end{figure}

\end{example}

\begin{example}
For the 2-bridge link $\L_{k}=b(4k^{2}+4k, -2k-1)=L_{1}\cup L_{2}$, consider the cable link $\L_{cab}=L_{(p_{1}, q_{1})}\cup L_{(p_{2}, q_{2})}$ where $p_{i}, q_{i}$ are coprime positive integers with $q_{i}/p_{i}$ sufficiently large. By proposition \ref{cable sharp}, $\G_{HF}(\L_{cab})=\G(\L_{cab})$ is shown as in Figure \ref{cable of bridge} where all the horizontal segments in the ``stair" have length $p_{1}$, vertical segments have length $p_{2}$, and there are $k$ steps.

\end{example}

\begin{figure}[H]
\centering
\includegraphics[width=2.5in]{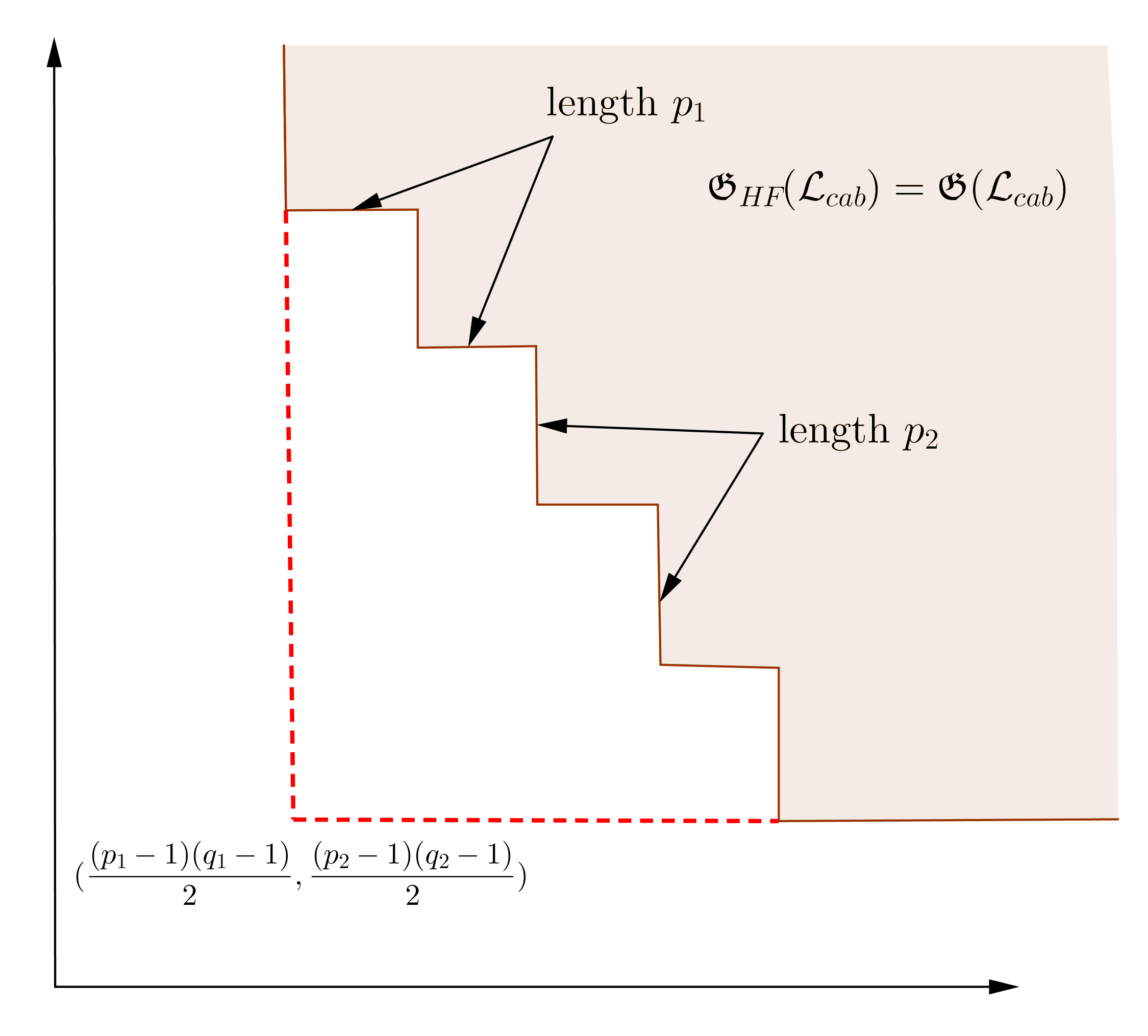}

\caption{The set $\G(\L_{cab})$ for the cable link   \label{cable of bridge}}

\end{figure}

\end{document}